\documentclass[11pt,oneside]{amsart}
\usepackage[T1]{fontenc}
\usepackage[latin9]{inputenc}
\usepackage{varioref}
\usepackage{amsthm}
\usepackage{amsbsy}
\usepackage{amstext}
\usepackage{amssymb}
\usepackage{esint}
\usepackage[unicode=true,pdfusetitle,
 bookmarks=true,bookmarksnumbered=false,bookmarksopen=false,
 breaklinks=false,pdfborder={0 0 1},backref=false,colorlinks=false]
 {hyperref}
\hypersetup{
 pdfencoding=auto}
\usepackage{breakurl}

\makeatletter

\newcommand{\noun}[1]{\textsc{#1}}

\numberwithin{equation}{section}
\numberwithin{figure}{section}
  \theoremstyle{plain}
  \newtheorem*{thm*}{\protect\theoremname}
  \theoremstyle{plain}
  \newtheorem*{prop*}{\protect\propositionname}
\theoremstyle{plain}
\newtheorem{thm}{\protect\theoremname}
  \theoremstyle{plain}
  \newtheorem{prop}[thm]{\protect\propositionname}
  \theoremstyle{definition}
  \newtheorem{defn}[thm]{\protect\definitionname}
  \theoremstyle{remark}
  \newtheorem{rem}[thm]{\protect\remarkname}
  \theoremstyle{definition}
  \newtheorem{example}[thm]{\protect\examplename}
  \theoremstyle{plain}
  \newtheorem{cor}[thm]{\protect\corollaryname}
  \theoremstyle{remark}
  \newtheorem{acknowledgement}[thm]{\protect\acknowledgementname}

\usepackage[a4paper,margin=2cm]{geometry}
\usepackage{mathtools}
\usepackage{tikz}
\usetikzlibrary{decorations.pathreplacing}
\usepackage{breqn}

\makeatother

  \providecommand{\acknowledgementname}{Acknowledgement}
  \providecommand{\corollaryname}{Corollary}
  \providecommand{\definitionname}{Definition}
  \providecommand{\examplename}{Example}
  \providecommand{\propositionname}{Proposition}
  \providecommand{\remarkname}{Remark}
  \providecommand{\theoremname}{Theorem}
\providecommand{\theoremname}{Theorem}

\begin{document}

\title{an extension of the functional ito formula under a family of non--dominated
measures}
\begin{abstract}
Motivated by questions arising in financial mathematics, Dupire \cite{Duprire_FIto}
introduced a notion of smoothness for functionals of paths (different
from the usual Fréchet--Gatéaux derivatives) and arrived at a generalization
of It\={o}'s formula applicable to functionals which have a pathwise
continuous dependence on the trajectories of the underlying process.
We study nonlinear functionals which do not have such pathwise continuity
and further work simultaneously under the family of continuous semimartingale
measures on path-space. We do this without introducing a second component,
as carried out by Cont--Fournie \cite{ContFournie,ContFournie_Pathwise}
but by using old work of Bichteler \cite{bichteler1981stochastic}
which allows to keep a pathwise picture even for complex functionals.

\end{abstract}

\author{Harald Oberhauser}

\address{TU Berlin, Institut für Mathematik, Straße des 17.~Juni 136, 10623
Berlin}


\email{h.oberhauser@gmail.com}

\maketitle

\section{Introduction}

It is an understatement to say that It\={o}'s stochastic calculus
is a useful tool for the modeling of systems that evolve under the
influence of randomness. An essential part of this calculus is It\={o}'s
formula which, given a semimartingale $X=\left(X_{t}\right)_{t\geq0}$
and a sufficiently smooth function $f$, shows that the stochastic
process $f\left(t,X_{t}\right)$ is also a semimartingale and moreover,
provides an explicit (Bichteler--Dellacherie) decomposition of the
process $f\left(t,X_{t}\right)$ into a sum of a stochastic integral
against $X$ and a process of finite variation. Of course, the class
of processes of the form $f\left(t,X_{t}\right)$ is only a small
subset of the stochastic processes that are adapted to the filtration
generated by $X$, $\sigma\left(X\right)$, and often it is necessary
to derive similar statements for this larger class of processes. One
area where such questions arise is financial engineering: traders
aim to understand the dynamics of the price process of a (highly path-dependent)
contingent claim with respect to the underlying asset modeled by the
stochastic process $X$. Motivated by these questions, Dupire \cite{Duprire_FIto}
showed that if $X$ is a standard real-valued Brownian motion, it
is possible to extend It\={o}'s formula to a nontrivial subset in
the class of real-valued processes $F=\left(F_{t}\right)_{t\geq0}$
which are adapted to $\sigma\left(X\right)$ and gave the formula
\begin{equation}
F_{t}-F_{0}=\int_{0}^{t}\left(\partial_{0}F_{r}+\frac{1}{2}\Delta F_{r}\right)dr+\int_{0}^{t}\nabla F_{r}dX\label{eq:dito}
\end{equation}
where $\partial_{0}F$,$\nabla F$ and $\Delta F$ are again $\sigma\left(X\right)$-adapted
processes to which we refer as the causal derivatives of $F$ (following
M.~Fliess%
\footnote{We use the term causal since essentially the same operators can be
found in the literature on nonlinear system control (i.e.~deterministic,
bounded variation paths), c.f.~\cite[Section IIb]{fliess1983concept},
where they are known as causal derivatives.%
}; in \cite{Duprire_FIto,ContFournie,ContFournie_Pathwise} these are
called functional time and space derivatives or also horizontal and
vertical derivatives). For the special case $F_{t}=f\left(t,X_{t}\right)$,
these causal derivatives coincide with the usual time and space derivatives
of $f$ and the above reduces to the standard It\={o}-formula applied
to $f\left(t,X_{t}\right)$. However, the class of processes considered
by Dupire is limited to those that depend continuously%
\footnote{in (not precisely but) essentially the uniform topology, cf.~Example
\ref{ex: continuous proc} and \cite{ContFournie}; in fact, much
of this article is concerned with avoiding the use of this metric.%
} on the trajectories of $X$ (e.g.~$F_{.}=\int_{0}^{.}f\left(r,X_{r}\right)dr$
but not $F_{.}=\int_{0}^{.}f\left(r,X_{r},\left[X\right]_{r}\right)dr$
etc.). While such a path-by-path continuity is of course guaranteed
by differentiability in the case of the classic Ito-formula, $F_{t}=f\left(t,X_{t}\right)$,
it is not implied for general, $\sigma\left(X\right)$-adapted processes
$F$ by causal differentiability; indeed, such a pathwise continuity
would be a very strong restriction in the class of $\sigma\left(X\right)$-adapted
processes. This restriction was subsequently addressed by Cont--Fournie
who consider general semimartingales $X$ and add a second process
(e.g.~$\left[X\right]$ or its weak derivative; cf.~\cite{ContFournie,ContFournie_Pathwise}
and the excellent thesis of D.~Fournie) to express $F$ as a functional
which depends pathwise continuously on the trajectories of $X$ \emph{and}
this second process which allows them to arrive at a generalization
of Dupire's functional It\={o}-formula to a larger class of $\sigma\left(X\right)$-adapted
processes. This article is inspired by all these strong and beautiful
results and shares the same goal of extending the class of functionals
of $X$ for which (\ref{eq:dito}) holds, however, we explore a different
approach. We do not introduce a second process in addition to the
continuous semimartingale $X$ but nevertheless arrive at a generalisation
of (\ref{eq:dito}) to a large class of functionals which includes
adapted processes with a complex dependence on the past like the quadratic
variation of $X$, stochastic integrals, Doléans--Dade exponentials,
compositions thereof, etc. To do this we have to overcome some obstacles: 
\begin{itemize}
\item firstly, the pathwise nature of the causal derivatives involves operations
on sets of paths which are null-sets,
\item secondly, many interesting processes are constructed by a probabilistic
argument, hence are often only uniquely defined as equivalence classes
of indistinguishable processes but not pathwise unique,
\item thirdly, that many of these processes are not robust under approximations
of $X$ in uniform norm.
\end{itemize}
Especially note the clash of the first point (pathwise considerations
matter) and the second point (only equivalence classes under a fixed
probability measure matter for many applications in stochastic analysis).
We briefly elaborate on this and sketch our approach: Throughout we
work on the canonical path-space $\left(\Omega,\mathcal{F}^{0},\left(\mathcal{F}_{t}^{0}\right)\right)$
of $\mathbb{R}^{d}$-valued (possibly defective, cf.~Section \ref{sub:Notation.})
càdlàg processes with $X_{t}\left(\omega\right)=\omega\left(t\right)$
denoting the coordinate process. In Section \ref{sec:Causal-derivatives-of}
we recall Dupire's causal derivatives. They are a family of maps $\mathbb{R}_{+}\times\Omega\rightarrow\mathbb{R}$
which describe the sensitivity of a given $F:\mathbb{R}_{+}\times\Omega\rightarrow\mathbb{R}$
with respect to perturbations of the coordinate process $X$ at running
time. If we denote with $\mathcal{M}_{c}^{semi}$ the probability
measures under which the coordinate process $X$ is a semimartingale
with (\noun{$\mathbb{P}$}-a\noun{.}s) continuous trajectories then
we can consider the standard completion $\left(\Omega,\mathcal{F}^{\mathbb{P}},\mathcal{F}_{t}^{\mathbb{P}},\mathbb{P}\right)$
of $\left(\Omega,\mathcal{F}^{0},\mathcal{F}_{t}^{0}\right)$ (the
reason why we work on the space of càdlàg paths $\Omega$ is the definition
of the causal space derivative). Now assume we are given an adapted,
real-valued process $F$ on $\left(\Omega,\mathcal{F}^{\mathbb{P}},\mathcal{F}_{t}^{\mathbb{P}},\mathbb{P}\right)$,
\emph{$\mathbb{P}\in\mathcal{M}_{c}^{semi}$} then even if we have
a version of this $F$ at hand which has causal derivatives then it
is, as pointed out above, not always justified to make strong assumptions
on the pathwise regularity of the map $\left(t,\omega\right)\mapsto F_{t}\left(\omega\right)$
(already for fixed $t$; e.g.~if $F$ is any version of the stochastic
integral, quadratic variation etc.). Therefore we introduce in Section
\ref{sec:regular-and-differentiable} the class of functionals $C^{1,2}$,
i.e.~a subset of the maps $\mathbb{R}_{+}\times\Omega\rightarrow\mathbb{R}$,
and show in Section \ref{sec:An-ito-Formula} that the functional
It\={o} formula extends to $C^{1,2}$; Theorem \vref{thm:pathdependent ito}
reads
\begin{thm*}
Let $F:\mathbb{R}_{+}\times\Omega\rightarrow\mathbb{R}$, $F\in C^{1,2}$.
Then $\forall\mathbb{P}\in\mathcal{M}_{c}^{semi}$, $F$ is a continuous
semimartingale on $\left(\Omega,\mathcal{F}^{\mathbb{P}},\mathcal{F}_{t}^{\mathbb{P}},\mathbb{P}\right)$
and 
\[
F_{.}-F_{0}=\int_{0}^{.}\partial_{0}F_{r}dr+\sum_{i=1}^{d}\int_{0}^{.}\partial_{i}F_{r-}dX_{r}^{i}+\frac{1}{2}\sum_{i,j=1}^{d}\int_{0}^{.}\partial_{ij}F_{r-}d\left[X^{i},X^{j}\right]_{r}\,\,\,\,\mathbb{P}-a.s.
\]

\end{thm*}
Of course, it remains to show that an interesting class of processes
can be expressed as a functional in $C^{1,2}$: for the class of processes
studied by Dupire (Example \vref{ex: continuous proc}) this follows
immediately. In Section \ref{sec:A-regular-and} we do exactly this
for the stochastic integral $\int Y_{-}dX$ and the quadratic variation
process $\left[X^{i},X^{j}\right]$ using the pathwise It\={o}-integral
due to Bichteler and Karandikar, for example we show in Section \ref{sec:A-regular-and}
\begin{prop*}
For every $i,j\in\left\{ 1,\ldots,d\right\} $ there exists a map
$B^{ij}:\mathbb{R}_{+}\times\Omega\rightarrow\mathbb{R}$ such that
$B^{ij}\in C^{1,2}$ and $B^{ij}$ is on every $\left(\Omega,\mathcal{F}^{\mathbb{P}},\mathcal{F}_{t}^{\mathbb{P}},\mathbb{P}\right)$,
$\mathbb{P\in\mathcal{M}}_{c}^{semi}$, an adapted process, indistinguishable
from the quadratic variation process $\left[X^{i},X^{j}\right]$ (constructed
under $\mathbb{P}$).
\end{prop*}
Note that this is a non--trivial statement linking the map/functional
$B^{ij}$ with the quadratic variation process $\left[X^{i},X^{j}\right]$,
the latter being only uniquely defined as an equivalence class of
indistinguishable processes on $\left(\Omega,\mathcal{F}^{\mathbb{P}},\mathcal{F}_{t}^{\mathbb{P}},\mathbb{P}\right)$
for every fixed \noun{$\mathbb{P}\in\mathcal{M}_{c}^{semi}$}. The
difficulty is to choose a representant in this class which is an element
of $C^{1,2}$ and additionaly also a representant for the quadratic
variation constructed under any other measure in $\mathcal{M}_{c}^{semi}$,
i.e.~to find an aggregator%
\footnote{Given a family $\mathcal{P}$ of probability measures and a family
of processes $\left\{ F^{\mathbb{P}}:\mathbb{P}\in\mathcal{P},F^{\mathbb{P}}\text{ is a meas.\ process on }\left(\Omega,\mathcal{F}^{\mathbb{P}},\mathcal{F}_{t}^{\mathbb{P}}\right)\right\} $
we call $F:\mathbb{R}_{+}\times\Omega\rightarrow\mathbb{R}$ an aggregator
of this family if $F=F^{\mathbb{P}}$ $\mathbb{P}$-a.s\@.~$\forall\mathbb{P}\in\mathcal{P}$
(see \cite{EJP2011-67} for a slightly different definition).%
} which is also differentiable --- if we would restrict ourselves to
a subset $\mathcal{P}\subset\mathcal{M}_{c}^{semi}$ which is dominated
by a single measure, i.e.~all elements in $\mathcal{P}$ are absolutely
continuous wrt to this measure, then at least the existence of the
aggregator would be trivial though not necessarily its differentiability.
A similar result holds for the It\={o}--integral $\int Y_{-}\cdot dX$
and as a consequence of the results in Section \ref{sec:A-regular-and}
we get
\begin{prop*}
Let $\mu,\sigma:\mathbb{R}_{+}\times\Omega\rightarrow\mathbb{R}$
be sufficiently regular (as in Theorem \ref{thm:smooth stoch integral}
resp.~Example \ref{ex:regular ito process}). Then there exists a
map $I:\mathbb{R}_{+}\times\Omega\rightarrow\mathbb{R}$ such that
$I\in C^{1,2}$ and $I$ is on every $\left(\Omega,\mathcal{F}^{\mathbb{P}},\mathcal{F}_{t}^{\mathbb{P}},\mathbb{P}\right)$,$\mathbb{P\in\mathcal{M}}_{c}^{semi}$,
an adapted process, indistinguishable from the process (constructed
under $\mathbb{P}$) 
\begin{equation}
\int_{0}^{.}\mu_{r}dr+\int_{0}^{.}\sigma_{r-}dX_{r}.\label{eq:gIto}
\end{equation}
Moreover, $\partial_{0}F=\mu$, $\left(\partial_{i}F\right)_{i=1,\ldots,d}=\left(\sigma_{-}^{i}\right)_{i=1,\ldots,d}$
and $\left(\partial_{ij}F\right)_{i,j=1,\ldots,d}=0$ $\mathbb{P}$-a.s. 
\end{prop*}
If we denote with $\mathcal{P}_{ac}$ the subset of $\mathcal{M}_{c}^{semi}$
under which $t\mapsto\left[X\right]_{t}$ is absolutely continuous
wrt to Lebesgue measure then we can sum up the above in the language
of quasi--sure analysis \cite{denis2006theoretical,EJP2011-67}: \emph{$F=\overline{F}$
}$\mathcal{P}_{ac}$--quasi-sure for some $\overline{F}\in C^{1,2}$
iff $F$ is of the form (\ref{eq:gIto}) $\mathcal{P}_{ac}$--quasi--sure%
\footnote{An event holds $\mathcal{P}$-quasi-sure if it holds $\mathbb{P}$-a.s.~$\forall\mathbb{P}\in\mathcal{P}$
where $\mathcal{P}$ is a subset of $\mathcal{M}_{c}^{semi}$ (cf.~\cite{denis2006theoretical,MR2179357FS}).%
} (under suitable assumptions on the coefficients $\mu,\sigma$); further,
the class $C^{1,2}$ contains aggregators with complex pathdependence.
We emphasize that we do not extend the causal derivatives by closure
of operators on equivalence classes of indistinguishable processes
to cover It\={o}--processes which leads to strong results under one
fixed measure $\mathbb{P}$ but develop a pathwise picture applicable
simultaneously to all $\mathbb{P}\in\mathcal{M}_{c}^{semi}$. The
approach to consider a stochastic calculus under a family of non-dominated
probability measures is motivated by recent developments in probability
theory like Denis' and Martini's uncertain volatility model and quasi-sure
analysis via aggregation \cite{EJP2011-67,denis2006theoretical} as
well as Peng's $G$-expectation \cite{peng2010nonlinear,2011arXiv1108.4317P}.
Of course, plugging in above functionals $B$ and $I$ into the functional
It\={o} formula leads to trivial identities but the point is that
compositions with smooth functions or functionals are again in $C^{1,2}$
(e.g.~see the examples in Section \ref{sec:A-regular-and} or \cite{Fourniethesis}
for more applications in finance). Further, the only requirement for
a functional $F:\mathbb{R}_{+}\times\Omega\rightarrow\mathbb{R}$
for being in $C^{1,2}$ is besides its causal differentiability (and
some minor regularity assumptions) the convergence of $F$ (and its
derivatives) under finite--dimensional approximations, $F\left(\omega^{n}\right)\rightarrow_{n}F\left(\omega\right)$
uniformly on compacts in $\mathbb{P}$-probability, $\forall\mathbb{P}\in\mathcal{M}_{c}^{semi}$,
where $\left(\omega^{n}\right)_{n}$ is a piecewise constant approximation
to $\omega$, which --- in view of the usual approximation results
in It\={o}--calculus --- seems to be quite a weak and natural assumption
(recall classic approximations in It\={o}--calculus for say the stochastic
integral or quadratic variation which more or less by construction
hold in probability or even Wong--Zakai type%
\footnote{For obvious reasons, Wong--Zakai results are usually formlated for
the convergence along piecewise linear paths (not piecewise constant)
hence giving a Stratonovich- (not an It\={o}-) calculus.%
} results for highly complex, pathdependent functionals are in principle
included).

\subsection{Notation.\label{sub:Notation.}}

Denote with $\zeta$ an isolated point added to $\mathbb{R}^{d}$
which plays the role of a cemetery and denote with $\mathbb{R}_{+}$
the set $\left[0,\infty\right)$. We work on the canonical path-space
for sub-stochastic (i.e.~possibly defective) càdlàg processes denoted
with $\Omega$, i.e.~$\Omega$ is the set of paths $\omega:\mathbb{R}_{+}\rightarrow\mathbb{R}^{d}\cup\left\{ \zeta\right\} $
with lifetime
\[
\zeta\left(\omega\right)=\inf\left\{ t\geq0:\omega\left(t\right)=\zeta\right\} 
\]
which are càdlàg and stay at the cemetery $\zeta$ after their lifetime%
\footnote{With abuse of notation we use $\zeta$ for the isolated point as well
as the map $\zeta:\Omega\rightarrow\left[0,\infty\right]$. Since
$\left\{ \omega:\zeta\left(\omega\right)<t\right\} =\bigcup_{r<t,r\in\mathbb{Q}}\left\{ \omega:X_{r}\left(\omega\right)=\zeta\right\} \in\mathcal{F}_{t}^{0}$
it follows that $\zeta\left(.\right)$ is a random variable. In fact
the cemetery $\zeta$ is not essential for our arguments but it allows
for intuitive characterizations of predictable processes (Proposition
\ref{prop:canonical pathspace}) which is useful for some proofs.%
} $\zeta\left(\omega\right)$. Denote the coordinate process $\left(X_{t}\right)_{t\in\mathbb{R}_{+}}$,
\[
X_{t}\left(\omega\right):=\omega\left(t\right)
\]
as well as 
\[
X_{t-}\left(\omega\right):=\lim_{s\uparrow t}X_{s}\left(\omega\right),\text{ }\Delta_{t}X\left(\omega\right):=X_{t}\left(\omega\right)-X_{t-}\left(\omega\right).
\]
Further, we introduce a $\sigma$-field and filtration on $\Omega$,
\[
\mathcal{F}^{0}:=\sigma\left(X_{s},0\leq s\right)\text{ and }\mathcal{F}_{t}^{0}:=\sigma\left(X_{s},0\leq s\leq t\right).
\]
Note that $\Omega$ can be endowed with the Skorohod topology in which
case the Borel $\sigma$-algebra equals $\mathcal{F}^{0}$ and with
this topology $\Omega$ is a Polish space (cf.~\cite[Chapter VI, Theorem 1.4]{MR1943877}
or \cite[Chapter 0]{bertoin1998levy}; however we are not making use
of this Polish structure in this article). To speak about predictable
processes we have to define $\mathcal{F}_{0-}$ and $X_{0-}$ which
we simply define as $\mathcal{F}_{0-}=\left\{ \emptyset,\Omega\right\} $,
$X_{0-}=\zeta$; further denote the Borel $\sigma-$field of $\mathbb{R}_{+}$
with $\mathcal{B}_{\mathbb{R}_{+}}$ and the Lebesgue measure on $\mathbb{R}_{+}$
with $\lambda$. As usual we call any collection of random variables
on $\Omega$ indexed by time $t\in\mathbb{R}_{+}$ a stochastic process
and a stochastic process $\left(F_{t}\right)$ is said to be measurable
if $\left(t,\omega\right)\mapsto F_{t}\left(\omega\right)$ is measurable
on $\mathbb{R}_{+}\times\Omega$ with respect to $\mathcal{B}_{\mathbb{R}_{+}}\times\mathcal{F}^{0}$;
similarly $\left(F_{t}\right)$ is progressively measurable if for
each $t\in\mathbb{R}_{+}$ the map $\left[0,t\right]\times\Omega\rightarrow\mathbb{R}$,
$\left(s,\omega\right)\mapsto F_{s}\left(\omega\right)$ is $\left(\mathcal{B}_{\left[0,t\right]}\times\mathcal{F}_{t}^{0}\right)$-measurable
and a process $\left(F_{t}\right)$ is said to be adapted to $\left(\mathcal{F}_{t}^{0}\right)$
if for each $t\in\mathbb{R}_{+}$, $F_{t}$ is $\mathcal{F}_{t}^{0}$-measurable.
The optional $\sigma$-field $\mathcal{O}^{0}$ on $\mathbb{R}_{+}\times\Omega$
is generated by the real-valued càdlàg processes adapted to $\left(\mathcal{F}_{t}^{0}\right)_{t\geq0}$
and the predictable $\sigma$-field $\mathcal{P}^{0}$ on $\mathbb{R}_{+}\times\Omega$
is generated by the real-valued processes adapted to%
\footnote{this is equivalent except at $t=0$ to being adapted to $\left(\mathcal{F}_{t}^{0}\right)_{t\geq0}$
. To quote Dellacherie--Meyer \cite[p121-IV]{MR521810}, ``in all
considerations on the predictable $\sigma$-field, time $0$ plays
the devil's role''.%
} $\left(\mathcal{F}_{t-}^{0}\right)_{t\geq0}$ with càg (left-continuous)
paths on $\left(0,\infty\right)$, cf.~\cite[p121-IV]{MR521810}
for further properties of these $\sigma$-algebras. Given $\left(t,\omega\right)\in\mathbb{R}_{+}\times\Omega$
and $r\in\mathbb{R}^{d}\cup\left\{ \zeta\right\} $, we denote $\omega^{t,r}\in\Omega$
the càdlàg path which coincides until time $t$ with $\omega$ but
has a jump at time $t$ in direction $r\in\mathbb{R}^{d}\cup\left\{ \zeta\right\} $
and stays constant after time $t$, viz. 
\[
\omega^{t,r}\left(s\right)=\omega\left(s\wedge t\right)+r1_{s\geq t},\,\, s\in\mathbb{R}_{+}
\]
(with the convention $a+\zeta=\zeta$ for any $a\in\mathbb{R}^{d}\cup\left\{ \zeta\right\} $);
further denote with $\omega_{\wedge t}$ the path $\omega$ stopped
at time $t$, viz. 
\[
\omega_{\wedge t}\left(s\right)=\omega\left(s\wedge t\right),\,\, s\in\mathbb{R}_{+}.
\]
These two perturbations of the path $\omega$ play a central role
in this article and lead to the definition of a causal time and space
derivative of a process in Section \ref{sec:Causal-derivatives-of}.
However, at this point we only recall that such perturbations arise
naturally when working on the canonical path-space $\left(\Omega,\mathcal{F}^{0},\left(\mathcal{F}_{t}^{0}\right)\right)$
and lead to an intuitive characterization of predictable and optional
processes.
\begin{prop}[{Dellacherie--Meyer \cite[p147-IV]{MR521810}}]
\textup{\label{prop:canonical pathspace}~}
\begin{enumerate}
\item \textup{The predictable $\sigma$-field $\mathcal{P}^{0}$ on $\mathbb{R}_{+}\times\Omega$
is generated by the two maps
\[
\left(t,\omega\right)\mapsto t\text{ and }\left(t,\omega\right)\mapsto\omega^{t,\zeta}
\]
and the optional $\sigma$-field $\mathcal{O}^{0}$ on $\mathbb{R}_{+}\times\Omega$
is generated by the two maps 
\[
\left(t,\omega\right)\mapsto t\text{ and }\left(t,\omega\right)\mapsto\omega_{\wedge t}.
\]
}
\item \textup{A measurable process $F$ is predictable (wrt to $\mathcal{P}^{0}$)
iff 
\[
F_{t}\left(\omega\right)=F_{t}\left(\omega^{t,\zeta}\right)\,\,\,\forall\left(t,\omega\right)\in\mathbb{R}_{+}\times\Omega
\]
and a measurable process $F$ is optional (wrt to $\mathcal{O}^{0}$)
iff
\[
F_{t}\left(\omega\right)=F_{t}\left(\omega_{\wedge t}\right)\,\,\,\forall\left(t,\omega\right)\in\mathbb{R}_{+}\times\Omega.
\]
 }
\end{enumerate}
\end{prop}
We say that a sequence $\left(\pi\left(n\right)\right)_{n}$ of partitions
of \noun{$\mathbb{R}_{+}$} (i.e.~$\pi\left(n\right)=\left\{ t_{k}^{n}\right\} _{k\geq0}$
where $t_{k}^{n}\in\mathbb{R}_{+}$, $t_{k}^{n}\leq t_{k+1}^{n}$
and $t_{0}^{0}=0$) converges to the identity if 
\[
\sup_{k\geq0}\left|t_{k}^{n}-t_{k-1}^{n}\right|\rightarrow0\text{ and }\sup_{k\geq0}t_{k}^{n}\rightarrow+\infty\text{ as }n\rightarrow\infty.
\]
Denote for a given $t\in\mathbb{R}_{+}$ with $t^{\pi\left(n\right)}$
resp.~$t_{\pi\left(n\right)}$ the nearest elements of the partition
$\pi\left(n\right)$ to the right resp.~left of $t$ with the convention
$t\in\left[t_{\pi\left(n\right)},t^{\pi\left(n\right)}\right)$; similarly
we use $\ensuremath{\prescript{\pi\left(n\right)}{}{t}}$ resp.~$\ensuremath{\prescript{}{\pi\left(n\right)}{t}}$
with the convention $t\in\left(\ensuremath{\prescript{}{\pi\left(n\right)}{t}},\ensuremath{\prescript{\pi\left(n\right)}{}{t}}\right]$.
Further define for $\omega\in\Omega$ the path $\omega^{\pi\left(n\right)}\in\Omega$
as the piecewise constant approximation%
\footnote{With the convention $\zeta+r=\zeta$ for $r\in\mathbb{R}^{d}\cup\left\{ \zeta\right\} $.%
} of $\omega$ along $\pi\left(n\right)$ 
\begin{equation}
\omega^{\pi\left(n\right)}\left(t\right)=\omega\left(0\right)+\sum_{k:t_{k-1}^{n}\in\pi\left(n\right)}\left(\omega\left(t_{k}^{n}\right)-\omega\left(t_{k-1}^{n}\right)\right)1_{t\geq t_{k}^{n}}.\label{eq:omega_pi}
\end{equation}
Given a probability measure $\mathbb{P}$ denote with $\left(\Omega,\mathcal{F}^{\mathbb{P}},\mathcal{F}_{t}^{\mathbb{P}},\mathbb{P}\right)$\emph{
}the usual augmentation (i.e.~right-continuous and complete) of $\left(\Omega,\mathcal{F}^{0},\mathcal{F}_{t}^{0},\mathbb{P}\right)$.
Denote with $\mathcal{M}_{c}^{semi}$ the set of probability measures
$\mathbb{P}$ under which the coordinate process $X$ is a semimartingale
with $\mathbb{P}$-a.s.~continuous trajectories and similarly with
$\mathcal{M}_{c}$ resp.~$\mathcal{M}_{c}^{loc}$ the space of measures
under which $X$ is a martingale resp.~locale martingale with $\mathbb{P}$-a.s.~trajectories.
Above objects and terminology (measurable, the optional and predictable
$\sigma$-fields, etc.) can be defined with respect to the completed
$\sigma$-algebra $\mathcal{F}^{\mathbb{P}}$ resp.~$\left(\mathcal{F}_{t}^{\mathbb{P}}\right)$
instead of $\mathcal{F}^{0}$ and $\left(\mathcal{F}_{t}^{0}\right)$,
e.g.~giving rise to the predictable $\sigma$-field $\mathcal{P}^{\mathbb{P}}$,
the optional $\sigma$-field $\mathcal{O}^{\mathbb{P}}$ (instead
of $\mathcal{O}^{0},\mathcal{P}^{0}$) etc. Recall that two processes
$F$ and $G$ are said to be indistinguishable/a version of each other
on $\left(\Omega,\mathcal{F}^{\mathbb{P}},\mathcal{F}_{t}^{\mathbb{P}},\mathbb{P}\right)$
if $\mathbb{P}\left(F_{t}=G_{t}\,\,\forall t\in\mathbb{R}_{+}\right)=1$;
since parts of this article require us to argue pathwise (i.e.~with
a fixed version rather than modulo indistinguishability) we decided
to be rather pedantic about spelling out quantifiers for all statements.

\section{\label{sec:Causal-derivatives-of}Causal functionals and causal derivatives}
\begin{defn}
We say that a map $F:\mathbb{R}_{+}\times\Omega\rightarrow\mathbb{R}$
is a causal functional if 
\[
F_{t}\left(\omega_{\wedge t}\right)=F_{t}\left(\omega\right)\,\,\,\,\forall\left(t,\omega\right)\in\mathbb{R}_{+}\times\Omega.
\]
\end{defn}
\begin{rem}
We do not require measurability of causal functionals with respect
to the uncompleted $\sigma$-algebra $\left(\mathcal{B}_{\mathbb{R}_{\geq0}}\times\mathcal{F}^{0}\right)$;
in other words not every causal functional $F$ is a process on $\left(\Omega,\mathcal{F}^{0},\mathcal{F}_{t}^{0}\right)$
but if $F$ is causal and $\left(\mathcal{B}_{\mathbb{R}_{\geq0}}\times\mathcal{F}^{0}\right)$-measurable
then $F$ is an $\mathcal{O}^{0}$-optional process on the canonical
path-space $\left(\Omega,\mathcal{F}^{0},\mathcal{F}_{t}^{0}\right)$
(by Proposition \ref{prop:canonical pathspace}, see also Remark \ref{rem:no measurability}).
Of course, every $\mathcal{O}^{0}$-optional process is a causal functional.
\end{rem}
Following Dupire \cite{Duprire_FIto} (see also earlier work of Fliess
\cite{fliess1983concept,fliess1981fonctionnelles} in the context
of bounded variation paths from which took the term \emph{causal}),
we define the time and space derivatives of causal functionals. 
\begin{defn}
Let $F:\mathbb{R}_{+}\times\Omega\rightarrow\mathbb{R}$ be a causal
functional. If $\forall\left(t,\omega\right)\in\mathbb{R}_{+}\times\Omega$,
$t<\zeta\left(\omega\right)$ the map 
\[
\mathbb{R}^{d}\ni r\mapsto F_{t}\left(\omega^{t,r}\right)\in\mathbb{R}
\]
is continuously differentiable at $r=0$ then we denote with $\nabla F_{t}\left(\omega\right)=\left(\partial_{1}F_{t}\left(\omega\right),\ldots,\partial_{d}F_{t}\left(\omega\right)\right)$
its Jacobian and set $\nabla F_{t}\left(\omega\right)=0$ for $\left(t,\omega\right)$
with $t\geq\zeta\left(\omega\right)$. We then say that $F$ has $\nabla F:\mathbb{R}_{+}\times\Omega\rightarrow\mathbb{R}^{d}$
as causal space derivative. Similarly, we define the matrix-valued
second causal space derivative $\Delta F=\left(\partial_{i}\partial_{j}F\right)_{i,j=1}^{d}$
as the Hessian.
\begin{defn}
\label{def:time-derivative}Let $F:\mathbb{R}_{+}\times\Omega\rightarrow\mathbb{R}$
be a causal functional. If $\forall\left(t,\omega\right)\in\mathbb{R}_{+}\times\Omega$,
$t<\zeta\left(\omega\right)$ the map 
\[
\mathbb{R}_{+}\ni r\mapsto F_{t+r}\left(\omega_{\wedge t}\right)\in\mathbb{R}
\]
is continuous and has a right-derivative at $r=0$ then we denote
with $\partial_{0}F_{t}\left(\omega\right)$ this right derivative
and set $\partial_{0}F_{t}\left(\omega\right)=0$ if $t\geq\zeta\left(\omega\right)$.
If additionally $r\mapsto\partial_{0}F_{r}\left(\omega\right)$ is
Riemann integrable then we say that $F$ has the map $\partial_{0}F:\mathbb{R}_{+}\times\Omega\rightarrow\mathbb{R}$
as causal time derivative.
\end{defn}
\end{defn}
\begin{rem}
These derivatives can be seen as natural extension of the usual time
and space derivatives of functions, e.g.~if $F_{t}\left(\omega\right)=f\left(t,X\left(t,\omega\right)\right)$
and $f\in C^{1,1}\left(\mathbb{R}_{+}\times\mathbb{R},\mathbb{R}\right)$
then $\partial_{0}F_{t}\left(\omega\right)=\frac{df}{dt}\left(t,X\left(t,\omega\right)\right)$
and $\nabla F_{t}\left(\omega\right)=\frac{df}{dx}\left(t,X\left(t,\omega\right)\right)$
(if $t<\zeta\left(\omega\right)$). Further, it is easy to verify
the analogues of the standard rules of differentiation, i.e.~chain
and product rules, etc.
\begin{rem}
\label{rm: causal der. adapted}If $\omega$ and $\tilde{\omega}$
coincide on $\left[0,t\right]$ then $\nabla F_{r}\left(\omega\right)=\nabla F_{r}\left(\tilde{\omega}\right)$
$\forall r\leq t$ and the same statement is true for $\Delta F$
and $\partial_{0}F$, i.e.~the causal derivatives of a causal functional
are again causal functionals.
\begin{rem}
\label{rem:indistinguishable}These\emph{ causal derivatives are defined
pathwise, hence do not respect null-sets on $\left(\Omega,\mathcal{F}^{\mathbb{P}},\mathcal{F}_{t}^{\mathbb{P}},\mathbb{P}\right)$
}for a given $\mathbb{P}\in\mathcal{M}_{c}^{semi}$ (unlike the Fréchet
derivative in directions of the Cameron--Martin space if $\mathbb{P}$
is Gaussian, as used by Bismut, Malliavin et al.). It may happen that
two processes $F$ and $G$ are indistinguishable on \emph{$\left(\Omega,\mathcal{F}^{\mathbb{P}},\mathcal{F}_{t}^{\mathbb{P}},\mathbb{P}\right)$},
both have causal time and space derivatives but apriori there is no
reason why these causal derivatives should actually be indistinguishable
from each other%
\footnote{\label{fn:different version}e.g.~let $F_{t}\left(\omega\right)=0$,
$\tilde{F}_{t}\left(\omega\right)=1_{\Delta_{t}X\left(\omega\right)\neq0}$
and $\overline{F}_{t}\left(\omega\right)=c\Delta_{t}X\left(\omega\right)$
for some $c\in\mathbb{R}$. Then $F$,$\tilde{F}$ and $\overline{F}$
are indistinguishable on $\left(\Omega,\mathcal{F}^{\mathbb{P}},\mathcal{F}_{t}^{\mathbb{P}},\mathbb{P}\right)$
$\forall\mathbb{P}\in\mathcal{M}_{c}^{semi}$, $F$ has causal time
and space derivatives (all equal $0$), however $\tilde{F}$ does
not even have a causal space derivative and $\overline{F}$ has a
space derivative equal to the arbitrary chosen $c$! %
}. To sum up, versions matter and we need an additional regularity
assumption.
\end{rem}
\end{rem}
Note that throughout this section we avoided the language of probability/measure
theory and above remark shows how things can go wrong --- in fact
above remark might disturb our reader because it implies that we have
to drop the elegant, usual approach of stochastic analysis to regard
two processes as equal if they are indistinguishable under the measure
$\mathbb{P}$. Instead we work pathwise. We hope to reconcile her
in Section \ref{sec:An-ito-Formula} where it turns out that under
an additional assumption on $F$ (namely regularity as introduced
in the section below \emph{and} causal differentiability) the examples
in Remark \ref{rem:indistinguishable} are to a large extent excluded
(Proposition \ref{prop:uniqueness}). In the mean time we simply ask
for her patience and mention that situations in stochastic analysis
in which indistinguishability is not a sufficient criteria are not
too uncommon: prominent examples include Clark's robustness problem
in nonlinear filtering \cite{ClarkRobustness}, quasi--sure analysis
via aggregation \cite{denis2006theoretical,EJP2011-67}, pathwise
expansions \cite{buckdahn2002pathwise} or the recent interest in
pathwise delta-hedging arguments where in all these cases the existence
of a ``good pathwise version'' is important.
\end{rem}

\section{\label{sec:regular-and-differentiable}Regular and differentiable
causal functionals}

We have to cope with the problem that functionals on an infinite-dimensional
(path-)space can be causal differentiable but not continuous with
respect to uniform topology. A simple solution is to additionally
require such a continuity, however, since we are interested in functionals
of unbounded variation paths this is indeed a very strong assumption.
This section introduces a class of functionals, regular enough to
prove a functional It\={o} formula but still rich enough to include
aggregators of the stochastic integral, quadratic variation, Doléans--Dade
exponential etc.

\subsection{Regular, causal functionals}

The proof of the functional It\={o} formula requires to understand
the behaviour of causal functionals under pathwise approximations
of the coordinate process $X$. We briefly recall that pathwise approximations
of functionals of unbounded variation paths are more subtle than the
case of bounded variation paths, already under a fixed measure $\mathbb{P}\in\mathcal{M}_{c}^{semi}$.
\begin{example}
\label{ex: Levy-area}Fix $\mathbb{P}\in\mathcal{M}_{c}^{semi}$ and
consider the stochastic process on $\left(\Omega,\mathcal{F},\mathcal{F}_{t}^{\mathbb{P}},\mathbb{P}\right)$
defined as 
\begin{eqnarray*}
F_{t}\left(\omega\right) & = & \left(\int_{0}^{t}X^{1}dX^{2}\right)\left(\omega\right)-\left(\int_{0}^{t}X^{2}dX^{1}\right)\left(\omega\right)
\end{eqnarray*}
(any version of the stochastic integral on the right hand side), i.e.~Lévy's
area process. To the best of our knowledge, the first explicit example
of an approximation of the process $X$ which leads to a correction
term is due to McShane \cite{MR0402921} for the case when $X$ is
a Brownian motion on $\left(\Omega,\mathcal{F}^{\mathbb{P}},\mathcal{F}_{t}^{\mathbb{P}},\mathbb{P}\right)$,
that is for an arbitrary $c\in\mathbb{R}$ there exists for $\mathbb{P}$-a.e.~$\omega$
a sequence $\left(\omega^{n}\right)_{n}\subset\Omega$ s.t.~$\left|X\left(\omega^{n}\right)-X\left(\omega\right)\right|_{\left[0,t\right],\infty}\rightarrow_{n}0$
but the process $F\left(\omega^{n}\right)$ converges to the process
$\left(F_{t}+c.t\right)_{t}$ $\mathbb{P}$-a.s.~--- we do not spell
out the details (but refer to the excellent presentation in \cite[Chapter VI]{MR0402921,ikeda-watanabe-89}
in the context of Wong--Zakai approximations) though we briefly recall
the main idea: consider the dyadic partitions and interpolate $\omega$
between two points not linear, but by choosing between two interpolating
functions to construct $\omega^{n}$. The key is to make this choice
dependent on the sign of (an approximation of) the Lévy area surpassed
in this time interval%
\footnote{McShane's and Sussman's approximations are not adapted but, like the
usual piecewise linear approximation, they can be simply shifted one
interval back in time to make them adapted.%
}. This highly-oscillatory perturbation is not canceled in the limit
as the mesh size tends to $0$ and picked up in an additive correction
term of the iterated integrals of $X$. Also note that such a behaviour
is not a question of Fisk--Stratonovich vs.~It\={o} integration (for
the Lévy-area $F$ it does not matter which integration is used since
the quadratic variation brackets cancel). Sussman (cf.~\cite{MR1119845,friz-oberhauser-2008b})
even shows that for arbitrary $N\in\mathbb{N}$ one can find approximations
such that the first $N$ iterated (Stratonovich) integrals converge
to the iterated integrals of $X$ but lead to a correction term on
the $\left(N+1\right)$-th level of iterated integrals. 
\end{example}
A classic strategy to cope with such phenomena is to restrict attention
to a class of ``natural'' finite-dimensional approximations, typically
piecewise constant or piecewise linear approximations (depending on
an It\={o} or a Stratonovich approach), e.g.~to prove Wong--Zakai
theorems \cite[Chapter VI]{wong1965convergence,eugene1965relation,ikeda-watanabe-89},
large deviation results à la Varadhan, Freidlin--Wentzell, Azencott
\cite{varadhan1984large,freuidlin1998random,azencott1980grandes},
or Stroock--Varadhan type support theorems \cite{MR532498} etc. In
view of an It\={o}-calculus and that the causal derivatives are defined
by piecewise constant perturbations resp.~perturbations by jumps,
it is not surprising that for us the ``natural class'' of approximations
are the piecewise constant approximations of the coordinate process
$X\left(\omega\right)=\omega$ (for the definition below recall that
$\omega^{\pi}$ is the piecwise constant approximation of $\omega$
along $\pi$; as in (\ref{eq:omega_pi})).
\begin{defn}
\label{def:regular, causal}We say a causal functional $F:\mathbb{R}_{+}\times\Omega\rightarrow\mathbb{R}$
is a weakly regular, càdlàg functional if for every $\mathbb{P}\in\mathcal{M}_{c}^{semi}$
there exists a sequence of partitions $\pi=\left(\pi\left(n\right)\right)_{n}$,
converging to identity such that 
\begin{enumerate}
\item $F$ is an \emph{$\left(\mathcal{F}_{t}^{\mathbb{P}}\right)$-}adapted
càdlàg process on%
\footnote{i.e.~$t\mapsto F_{t}$ is $\mathbb{P}$-a.s.~càdlàg%
} $\left(\Omega,\mathcal{F}^{\mathbb{P}},\mathcal{F}_{t}^{\mathbb{P}},\mathbb{P}\right)$,
\item the \emph{$\left(\mathcal{F}_{t}^{\mathbb{P}}\right)$-}adapted, measurable
process $\left(t,\omega\right)\mapsto F_{t}^{n}\left(\omega\right):=F_{t}\left(\omega^{\pi\left(n\right)}\right)$
has left limits (for every $n$ $\mathbb{P}$-a.s.) and $F_{-}^{n}\rightarrow F_{-}$
as $n\rightarrow\infty$ uniformly on compacts in $\mathbb{P}$-probability
(henceforth ucp%
\footnote{That is, $\forall T\in\mathbb{R}_{+}$ we have $\sup_{r\in\left[0,T\right]}\left|F_{r-}^{n}-F_{r-}\right|\rightarrow_{n\rightarrow\infty}0\text{ in probability on }\left(\Omega,\mathcal{F}^{\mathbb{P}},\mathcal{F}_{t}^{\mathbb{P}},\mathbb{P}\right).$%
}).
\end{enumerate}
If point (2) holds for every sequence of partitions $\pi=\left(\pi\left(n\right)\right)_{n}$
converging to identity and every $\mathbb{P}\in\mathcal{M}_{c}^{semi}$
then we simply say that $F$ is a regular, càdlàg functional. Similarly,
we say $F$ is a regular càglàd (resp.~continuous) functional if
in point (1) we replace the word càdlàg by càglàd (resp.~continuous).\end{defn}
\begin{rem}
Instead of point (2) we could simply demand that $F_{t}^{n}$ is a
càdlàg (or càglàd) process and $F^{n}\rightarrow_{n}F$ ucp. This
would already include a large class of functionals, however, the above,
more general formulation allows to include functionals $F$ which
only have left limits along the approximation $F^{n}$.
\begin{rem}
It is easy to write down a functional which is only weakly regular
but not regular. However, Section \ref{sec:A-regular-and} shows that
even processes with complex path-dependence have a ``canonical version''
which is regular, i.e.~a regular aggregator which is also causally
differentiable.
\end{rem}
\end{rem}
The important part of Definition \ref{def:regular, causal} is point
(2) since point (1) is by Proposition \ref{prop:canonical pathspace}
already satisfied whenever the causal functional is a measurable process
with càdlàg trajectories on the canonical probability space $\left(\Omega,\mathcal{F}^{0},\mathcal{F}_{t}^{0}\right)$.
An interesting subclass of functionals which are regular are the ones
treated by Dupire \cite{Duprire_FIto}.
\begin{example}
\label{ex: continuous proc}Let $F:\mathbb{R}_{+}\times\Omega\rightarrow\mathbb{R}$
be an $\mathcal{O}^{0}$-optional process on $\left(\Omega,\mathcal{F}^{0},\mathcal{F}_{t}^{0}\right)$
and as a map uniformly continuous on compacts of $\mathbb{R}_{+}$
with respect to the pseudo-metric on $\mathbb{R}_{+}\times\Omega\backslash\left\{ \omega:\zeta\left(\omega\right)<\infty\right\} $
\[
d_{Dupire}\left(\left(t,\omega\right),\left(\overline{t},\overline{\omega}\right)\right)=\left|t-\overline{t}\right|+\left|\omega_{\wedge t}-\overline{\omega}_{\wedge\overline{t}}\right|_{\infty;\left[0,t\vee\overline{t}\right]}.
\]
This is the class of processes studied by Dupire \cite{Duprire_FIto}.
Note that $\mathcal{O}^{0}\subset\mathcal{O}^{\mathbb{P}}$ and the
continuity with respect to $d_{Dupire}\left(.,.\right)$ guarantees
that $F$ is causal and also the convergence of $F_{t}\left(\omega^{\pi\left(n\right)}\right)$
as required by Definition \ref{def:regular, causal} (even pathwise!).
Hence, every such $F$ is a causal, regular continuous functional.
While this is quite a small class of processes in view of the usual
processes of interest in stochastic analysis, it already allows for
interesting examples (especially in mathematical finance where the
process $F$ can model the price evolution of a path-dependent option).\end{example}
\begin{rem}
\label{rem:no measurability}In Definition \ref{def:regular, causal}
of a regular functional we do not require $F$ to be a measurable
process on $\left(\Omega,\mathcal{F}^{0},\mathcal{F}_{t}^{0}\right)$
(by Proposition \ref{prop:canonical pathspace} only measurability
of $F$ is needed to make a causal functional an $\mathcal{O}^{0}$-optional
process) but only measurability on the completed probability spaces
$\left(\Omega,\mathcal{F}^{\mathbb{P}},\left(\mathcal{F}_{t}^{\mathbb{P}}\right),\mathbb{P}\right)$.
This is a minor technical point but this generality turns out to be
useful when dealing with causal functionals with very complex pathdependence
where measurability with respect to the uncompleted $\sigma$-algebra
$\mathcal{F}^{0}$ resp.~filtration $\left(\mathcal{F}_{t}^{0}\right)$
can be hard to establish (e.g.~in the pathwise Bichteler integral
where ``stopping/hitting times'' appear etc., see Section \ref{sec:A-regular-and}).
\end{rem}

\subsection{The class $C^{1,2}$ of causal functionals}

In this section, we introduce the class $C^{1,2}$ of causal functionals
which are regular and have regular, causal derivatives. Before we
do this we need another definition which will be only needed for the
causal derivatives of a functional.

\subsection{Causal continuity in time and space}
\begin{defn}
We say a causal functional $F:\mathbb{R}_{+}\times\Omega\rightarrow\mathbb{R}$
is causally continuous in space, if $\forall R>0$, $\forall t>0$
there exists a function $\rho_{t,R}:\mathbb{R}_{+}\rightarrow\mathbb{R}_{+}$
which is non-decreasing and $\lim_{x\searrow0}\rho_{t,R}\left(x\right)=0$
such that 
\[
\sup_{r\in\left[0,t\right]}\left|F_{r}\left(\omega^{r,\Delta_{r}}\right)-F_{r}\left(\omega^{r,\tilde{\Delta}_{r}}\right)\right|=\rho_{t,R}\left(\left|\Delta-\tilde{\Delta}\right|_{\infty;\left[0,t\right]}\right)
\]
holds $\forall\omega,\Delta,\tilde{\Delta}\in\Omega$ with $\left|\omega\right|_{\infty,\left[0,t\right]},\left|\Delta\right|_{\infty,\left[0,t\right]},\left|\tilde{\Delta}\right|_{\infty,\left[0,t\right]}\leq R$.
\end{defn}
Similarly, we give the corresponding definition in the time variable.
\begin{defn}
We say a causal functional $F:\mathbb{R}_{+}\times\Omega\rightarrow\mathbb{R}$
is causally continuous in time, if $\forall R>0$, $\forall t>0$
there exists a function $\rho_{t,R}:\mathbb{R}_{+}\rightarrow\mathbb{R}_{+}$
which is non-decreasing and $\lim_{x\searrow0}\rho_{t,R}\left(x\right)=0$
such that 
\[
\sup_{r\in\left[0,t\right]}\left|F_{r+\Delta_{r}}\left(\omega_{\wedge r}\right)-F_{r}\left(\omega\right)\right|=\rho_{t,R}\left(\left|\Delta\right|_{\infty;\left[0,t\right]}+\sup_{r\in\left[0,t\right]}\left|\omega_{r}-\omega_{r-}\right|\right)
\]
holds $\forall\omega\in\Omega$, $\forall\Delta\in C\left(\left[0,t\right],\mathbb{R}_{+}\right)$
with $\left|\omega\right|_{\infty;\left[0,t\right]},\left|\Delta\right|_{\infty;\left[0,t\right]}\leq R$.\end{defn}
\begin{rem}
To motivate above definitions note that they are immediately fulfilled
for $F_{t}\left(\omega\right)=f\left(t,X_{t}\left(\omega\right)\right)$
whenever $f\in C\left(\left[0,T\right]\times\mathbb{R}^{n},\mathbb{R}\right)$
(since $f$ is uniformly continuous on compacts, hence has a modulus).
Intuitively the continuity in space means that $F$ does react to
jumps in the underlying in proportion to the jump size and the continuity
in time implies that jumps of $F$ are only due to jumps in the underlying
(not due to progression of time alone). 
\begin{rem}
If $F$ is continuous with respect to $d_{Dupire}$ (Example \ref{ex: continuous proc})
then $F$ is causally continuous in time and space and the latter
are much weaker requirements for functionals than pathwise continuity
in the sense of Example \ref{ex: continuous proc} since we only require
continuity with respect to jumps resp.~stopping the path at running
time, not continuity with respect to perturbations of the whole past
of the path (e.g.~causal continuity in time and space hold for the
pathwise version of the stochastic integral or quadratic variation
in Section \ref{sec:A-regular-and}).
\end{rem}
\end{rem}

\subsection{The class $C^{1,2}$}

We now introduce the class of functionals we are interested in.
\begin{defn}
\label{def:C12}We denote with $C^{1,2}$ the set of all causal functionals
$F:\mathbb{R}_{+}\times\Omega\rightarrow\mathbb{R}$ which fulfill
\begin{enumerate}
\item $F$ is a regular continuous functional which has one causal time
and two causal space derivatives, 
\item $\partial_{0}F$,$\nabla F$,$\Delta F$ are regular càglàd or regular
càdlàg functionals,
\item $\nabla F,\Delta F$ are causally continuous in time and space.
\end{enumerate}
\end{defn}
\begin{rem}
For functions $f\left(t,X_{t}\right)$ continuity is guaranteed by
differentiability, however for causal functionals the questions of
causal differentiability and continuity wrt to the underlying have
to be treated separately. This is not surprising since we work with
nonlinear functionals on an infinite dimensional (path-)space.
\begin{rem}
We immediately see that $\left\{ \left(t,\omega\right)\mapsto f\left(t,\omega\left(t\right)\right):f\in C^{1,2}\left(\mathbb{R}_{+}\times\mathbb{R}^{d},\mathbb{R}\right)\right\} \subset C^{1,2}$,
i.e.~as our notation reveals, we think of $C^{1,2}$ as the analogue
for general $\sigma\left(X\right)$-adapted processes of the class
of processes $f\left(t,X_{t}\right)$ for $f\in C^{1,2}\left(\mathbb{R}_{+}\times\mathbb{R}^{d}\right)$.
Recall Example \ref{ex: continuous proc} which shows that differentiable
processes with pathwise continuous dependence on the underlying are
included in $C^{1,2}$. But most important for us, the results of
Section \ref{sec:A-regular-and} below show that $C^{1,2}$ includes
functionals which depend on the past of $X\left(\omega\right)=\omega$
in a more complex way like stochastic integrals, quadratic variation,
Doléans--Dade exponentials, compositions thereof, etc.
\end{rem}
We have arrived at a point comparable to the moment in a basic analysis
course where differentiability of functions has been introduced and
it remains to show that many basic functions are differentiable and
to calculate their derivative. In other words, we still have to show
that $C^{1,2}$ contains aggregators of interesting processes (besides
the examples covered by Dupire \cite{Duprire_FIto} resp.~Example
\ref{ex: continuous proc}). This trite observation is the motivation
for Section \ref{sec:A-regular-and}. However, before that, we show
in the section below that the functional It\={o}-formula extends to
the class of $C^{1,2}$ functionals.
\end{rem}

\section{\label{sec:An-ito-Formula}The functional it\={o} formula for $C^{1,2}$-functionals}

We are now in position to prove an extension of the functional It\={o}-formula
to $C^{1,2}$. The basic idea of the proof is similar to a standard
proof of the usual It\={o}-formula, which is to carry out a Taylor
expansion of order 2, and --- following Dupire --- the role of the
usual derivatives is taken by the causal time and space derivatives. 
\begin{thm}
\label{thm:pathdependent ito}If $F\in C^{1,2}$ then $F$ is a continuous
semimartingale on $\left(\Omega,\mathcal{F}^{\mathbb{P}},\mathcal{F}_{t}^{\mathbb{P}},\mathbb{P}\right)$
$\forall\mathbb{P}\in\mathcal{M}_{c}^{semi}$ and 
\begin{equation}
F_{.}-F_{0}=\int_{0}^{.}\partial_{0}F_{r}dr+\sum_{i=1}^{d}\int_{0}^{.}\partial_{i}F_{r-}dX_{r}^{i}+\frac{1}{2}\sum_{i,j=1}^{d}\int_{0}^{.}\partial_{ij}F_{r-}d\left[X^{i},X^{j}\right]_{r}\,\,\,\,\mathbb{P}-a.s.\label{eq:funct_tio}
\end{equation}
\end{thm}
\begin{proof}
Fix $t>0$ and a sequence $\left(\pi\left(n\right)\right)_{n}$ of
partitions $\pi\left(n\right)=\left(t_{k}^{n}\right)_{k}$ converging
to the identity. For brevity write $\omega^{n}$ for the piecewise
constant càdlàg approximation $\omega^{\pi\left(n\right)}$ of $\omega$
along $\pi\left(n\right)$ and $r^{n},\ensuremath{\prescript{n}{}{r}}$
resp.~$r_{n},\ensuremath{\prescript{}{n}{r}}$ for the nearest element
$\pi\left(n\right)$ to the right resp.~left of $r$ (with our usual
convention $r\in\left[r_{n},r^{n}\right)$ and $r\in\left(\ensuremath{\prescript{}{n}{r}},\ensuremath{\prescript{n}{}{r}}\right]$),
further we use the notation $X_{s,t}\left(\omega\right):=\omega_{s,t}:=\omega\left(t\right)-\omega\left(s\right)$,
$s,t\in\mathbb{R}_{+}$. Throughout assume that $n$ is big enough
s.t.~$t<\sup_{k}t_{k}^{n}$. For every $\omega\in\Omega$ with $\zeta\left(\omega\right)>t$
write 
\begin{eqnarray}
F_{t^{n}}\left(\omega^{n}\right)-F_{0}\left(\omega^{n}\right) & = & F_{t^{n}}\left(\omega_{\wedge t^{n}}^{n}\right)-F_{0}\left(\omega_{\wedge0}^{n}\right)\label{eq:F_t-F_0}\\
 & = & \sum_{k}F_{t_{k}^{n}}\left(\omega_{\wedge t_{k}^{n}}^{n}\right)-F_{t_{k-1}^{n}}\left(\omega_{\wedge t_{k-1}^{n}}^{n}\right)\nonumber 
\end{eqnarray}
(causality of $F$ implies $F_{t^{n}}\left(\omega^{n}\right)=F_{t^{n}}\left(\omega_{\wedge t^{n}}^{n}\right)$
and $F_{0}\left(\omega^{n}\right)=F_{0}\left(\omega_{\wedge0}^{n}\right)$)
where the sum $\sum_{k}$ is taken over all $k$ s.t.~$\pi\left(n\right)\ni t_{k-1}^{n},t_{k}^{n}\leq t^{n}$
and split the terms on the right hand side into 
\begin{eqnarray*}
F_{t_{k}^{n}}\left(\omega_{\wedge t_{k}^{n}}^{n}\right)-F_{t_{k-1}^{n}}\left(\omega_{\wedge t_{k-1}^{n}}^{n}\right) & = & \underbrace{F_{t_{k}^{n}}\left(\omega_{\wedge t_{k}^{n}}^{n}\right)-F_{t_{k}^{n}}\left(\omega_{\wedge t_{k-1}^{n}}^{n}\right)}_{=:S_{k}^{n}\left(\omega\right)}+\underbrace{F_{t_{k}^{n}}\left(\omega_{\wedge t_{k-1}^{n}}^{n}\right)-F_{t_{k-1}^{n}}\left(\omega_{\wedge t_{k-1}^{n}}^{n}\right)}_{=:T_{k}^{n}\left(\omega\right)}\text{.}
\end{eqnarray*}
The intuition is that $S_{k}^{n}$ measures the space increment at
time $t_{k}^{n}$ and $T_{k}^{n}$ measures the time decay on the
interval $\left[t_{k-1}^{n},t_{k}^{n}\right]$. In step 1 and 2 below
we show that for every $\mathbb{P}\in\mathcal{M}_{c}^{semi}$ we have
on $\left(\Omega,\mathcal{F}^{\mathbb{P}},\mathcal{F}_{t}^{\mathbb{P}},\mathbb{P}\right)$
the following convergence in $\mathbb{P}$-probability 
\begin{eqnarray}
\sum_{k}S_{k}^{n} & \rightarrow_{n\rightarrow\infty} & \int_{0}^{t}\nabla F_{r-}\cdot dX_{r}+\frac{1}{2}\int_{0}^{t}Tr\left[\Delta F_{r-}\cdot d\left[X\right]_{r}\right],\label{eq:sum space}\\
\sum_{k}T_{k}^{n} & \rightarrow_{n\rightarrow\infty} & \int_{0}^{t}\partial_{0}F_{r}dr.\label{sum time}
\end{eqnarray}
Further, since $F$ is a regular functional the process $\left(t,\omega\right)\mapsto F_{t-}\left(\omega^{n}\right)$
converges ucp to $F_{-}$ and since $t\mapsto F_{t}$ is $\mathbb{P}$-a.s.~continuous
(\ref{eq:F_t-F_0}) converges in probability to $F_{t}-F_{0}$. This
together with (\ref{eq:sum space}) and (\ref{sum time}) implies
(\ref{eq:funct_tio}) for fixed $t$. The continuity of $F$ then
allows to remove the time dependence of the null-set.

\textbf{Step} \textbf{1: }$\sum_{k}S_{k}^{n}\rightarrow_{n}^{\mathbb{P}}\int_{0}^{t}\nabla F_{-}dX+\frac{1}{2}\int_{0}^{t}Tr\left[\Delta F_{-}d\left[X\right]\right]$\textbf{.}
Fix $\omega\in\Omega$ with $\zeta\left(\omega\right)=\infty$ and
note that $\omega_{\wedge t_{k}^{n}}^{n}=\left(\omega_{\wedge t_{k-1}^{n}}^{n}\right)^{t_{k}^{n},\omega_{t_{k-1}^{n},t_{k}^{n}}}$.
Taylor's theorem applied to the function 
\[
\mathbb{R}^{d}\ni\bullet\mapsto F_{t_{k}^{n}}\left(\left(\omega_{\wedge t_{k-1}^{n}}^{n}\right)^{t_{k}^{n},\bullet}\right)\in\mathbb{R}
\]
guarantees the existence of a $\theta_{k}^{n}\left(\omega\right)\in\left[0,1\right]$
s.t. 
\begin{eqnarray*}
F_{t_{k}^{n}}\left(\omega_{\wedge t_{k}^{n}}^{n}\right)-F_{t_{k}^{n}}\left(\omega_{\wedge t_{k-1}^{n}}^{n}\right) & = & \sum_{i}\partial_{i}F_{t_{k}^{n}}\left(\omega_{\wedge t_{k-1}^{n}}^{n}\right)\omega_{t_{k-1}^{n},t_{k}^{n}}^{i}\\
 &  & +\frac{1}{2}\sum_{i,j}\partial_{ij}F_{t_{k}^{n}}\left(\left(\omega_{\wedge t_{k-1}^{n}}^{n}\right)^{t_{k}^{n},\theta_{k}^{n}\left(\omega\right)\omega_{t_{k-1}^{n},t_{k}^{n}}}\right)\omega_{t_{k-1}^{n},t_{k}^{n}}^{i}\omega_{t_{k-1}^{n},t_{k}^{n}}^{j}.
\end{eqnarray*}
Hence 
\begin{eqnarray}
F_{t_{k}^{n}}\left(\omega_{\wedge t_{k}^{n}}^{n}\right)-F_{t_{k}^{n}}\left(\omega_{\wedge t_{k-1}^{n}}^{n}\right) & = & \sum_{i}\partial_{i}F_{t_{k}^{n}}\left(\omega_{\wedge t_{k-1}}^{n}\right)\omega_{t_{k-1}^{n},t_{k}^{n}}^{i}\label{eq:spac_dec}\\
 &  & +\frac{1}{2}\sum_{i,j}\partial_{ij}F_{t_{k}^{n}}\left(\omega_{\wedge t_{k-1}^{n}}^{n}\right)\omega_{t_{k-1}^{n},t_{k}^{n}}^{i}\omega_{t_{k-1}^{n},t_{k}^{n}}^{j}+r_{n,k}\left(\omega,\omega_{t_{k-1}^{n},t_{k}^{n}}\right)\nonumber 
\end{eqnarray}
where
\begin{eqnarray}
r_{n,k}\left(\omega,\Delta\right) & = & \frac{1}{2}\sum_{i,j}\left[\partial_{ij}F_{t_{k}^{n}}\left(\left(\omega_{\wedge t_{k-1}^{n}}^{n}\right)^{t_{k}^{n},\theta_{k}^{n}\left(\omega\right)\Delta}\right)-\partial_{ij}F_{t_{k}^{n}}\left(\omega_{\wedge t_{k-1}^{n}}^{n}\right)\right]\Delta^{i}\Delta^{j}\,\,\,\forall\Delta\in\mathbb{R}^{d}.\label{eq:r_n_k}
\end{eqnarray}
By the causal continuity in space of $\Delta F=\left(\partial_{ij}F\right)_{i,j=1,\ldots,d}$
we can estimate (using that $\omega_{\wedge t_{k-1}^{n}}^{n}=\left(\omega^{n}\right)^{t_{k}^{n},-\omega_{t_{k-1}^{n},t_{k}^{n}}}$,$\left(\omega_{\wedge t_{k-1}^{n}}^{n}\right)^{t_{k}^{n},\theta_{k}^{n}\left(\omega\right)\Delta}=\left(\omega^{n}\right)^{t_{k}^{n},-\omega_{t_{k-1}^{n},t_{k}^{n}}+\theta_{k}^{n}\left(\omega\right)\Delta}$,$\left|\omega_{t_{k-1}^{n},t_{k}^{n}}\right|\leq2\left|\omega\right|,\left|\omega^{n}\right|\leq2\left|\omega\right|$)
\begin{eqnarray}
 &  & \left|\partial_{ij}F_{t_{k}^{n}}\left(\left(\omega_{\wedge t_{k-1}^{n}}^{n}\right)^{t_{k}^{n},\theta_{k}^{n}\left(\omega\right)\omega_{t_{k-1}^{n},t_{k}^{n}}}\right)-\partial_{ij}F_{t_{k}^{n}}\left(\omega_{\wedge t_{k-1}^{n}}^{n}\right)\right|\nonumber \\
 & = & \left|\partial_{ij}F_{t_{k}^{n}}\left(\left(\omega^{n}\right)^{t_{k}^{n},-\omega_{t_{k-1}^{n},t_{k}^{n}}+\theta_{k}^{n}\left(\omega\right)\omega_{t_{k-1}^{n},t_{k}^{n}}}\right)-\partial_{ij}F_{t_{k}^{n}}\left(\left(\omega^{n}\right)^{t_{k}^{n},-\omega_{t_{k-1}^{n},t_{k}^{n}}}\right)\right|\nonumber \\
 & \leq & \rho_{t,2\left|\omega\right|}\left(\left|\theta_{k}^{n}\left(\omega\right)\omega_{t_{k-1}^{n},t_{k}^{n}}\right|\right)\leq\rho_{t,2\left|\omega\right|}\left(\left|\omega_{t_{k-1}^{n},t_{k}^{n}}\right|\right)\label{eq:r_n_k estimate}
\end{eqnarray}
with $\rho_{t,2\left|\omega\right|}\left(0+\right)=0$. Taking $\sum_{k}$
in (\ref{eq:spac_dec}) yields 
\begin{eqnarray}
\sum_{k}S_{k}^{n}\left(\omega\right) & = & \sum_{k}\nabla F_{t_{k}^{n}}\left(\omega_{\wedge t_{k-1}^{n}}^{n}\right)\cdot X_{t_{k-1}^{n},t_{k}^{n}}\left(\omega\right)\label{eq:S_n_k}\\
 &  & +\frac{1}{2}\sum_{k}Tr\left[X_{t_{k-1}^{n},t_{k}^{n}}^{T}\left(\omega\right)\cdot\Delta F_{t_{k}^{n}}\left(\omega_{\wedge t_{k-1}^{n}}^{n}\right)\cdot X_{t_{k-1}^{n},t_{k}^{n}}\left(\omega\right)\right]\nonumber \\
 &  & +R_{n}\left(\omega\right)\nonumber 
\end{eqnarray}
with $R_{n}\left(\omega\right):=\sum_{k}r_{n,k}\left(\omega,\omega_{t_{k-1}^{n},t_{k}^{n}}\right)$
and by (\ref{eq:r_n_k}) and (\ref{eq:r_n_k estimate}) we arrive
at 
\begin{eqnarray*}
\left|R_{n}\left(\omega\right)\right| & \leq & \frac{1}{2}\sum_{k}\sum_{i,j}\left|\partial_{ij}F_{t_{k}^{n}}\left(\left(\omega_{\wedge t_{k-1}^{n}}^{n}\right)^{t_{k}^{n},\theta_{k}^{n}\left(\omega\right)\omega_{t_{k-1}^{n},t_{k}^{n}}}\right)-\partial_{ij}F_{t_{k}^{n}}\left(\omega_{\wedge t_{k-1}^{n}}^{n}\right)\right|\left|\omega_{t_{k-1}^{n},t_{k}^{n}}^{i}\omega_{t_{k-1}^{n},t_{k}^{n}}^{j}\right|\\
 & \leq & \frac{1}{2}\sup_{k}\rho_{t,2\left|\omega\right|}\left(\left|\omega_{t_{k-1}^{n},t_{k}^{n}}\right|\right)\sum_{k}\sum_{i,j}\left|X_{t_{k-1}^{n},t_{k}^{n}}^{i}\left(\omega\right)\right|\left|X_{t_{k-1}^{n},t_{k}^{n}}^{j}\left(\omega\right)\right|\\
 & \leq & \frac{1}{2}\rho_{t,2\left|\omega\right|}\left(\sup_{k}\left|\omega_{t_{k-1}^{n},t_{k}^{n}}\right|\right)\sum_{i,j}\sum_{k}\left(\left|X_{t_{k-1}^{n},t_{k}^{n}}^{i}\left(\omega\right)\right|^{2}+\left|X_{t_{k-1}^{n},t_{k}^{n}}^{j}\left(\omega\right)\right|^{2}\right).
\end{eqnarray*}
Note that $\sum_{k}\left(\left|X_{t_{k-1}^{n},t_{k}^{n}}^{i}\left(\omega\right)\right|^{2}+\left|X_{t_{k-1}^{n},t_{k}^{n}}^{j}\left(\omega\right)\right|^{2}\right)\rightarrow_{n}\left[X^{i}\right]_{t}+\left[X^{j}\right]_{t}$
in probability and obviously for $\mathbb{P}$-a.e.~$\omega$ we
have 
\[
\rho_{t,2\left|\omega\right|}\left(\sup_{k}\left|\omega_{t_{k-1}^{n},t_{k}^{n}}\right|\right)\rightarrow_{n}0
\]
hence $\left|R_{n}\left(.\right)\right|\rightarrow_{n}0$ in probability.
It remains to show convergence of the first two sums on the rhs in
(\ref{eq:S_n_k}). Therefore we rewrite $\sum_{k}\nabla F_{t_{k}^{n}}\left(\omega_{\wedge t_{k-1}^{n}}^{n}\right)\cdot X_{t_{k-1}^{n},t_{k}^{n}}\left(\omega\right)$
by adding and subtracting $\nabla F_{t_{k-1}^{n}}\left(\omega_{\wedge t_{k-1}^{n}}^{n}\right)$
and see that it is sufficient to prove 
\begin{eqnarray*}
\sum_{k}\nabla F_{t_{k-1}^{n}}\left(\omega_{\wedge t_{k-1}^{n}}^{n}\right)\cdot X_{t_{k-1}^{n},t_{k}^{n}}\left(\omega\right) & \rightarrow_{n\rightarrow\infty}^{\mathbb{P}} & \left(\int_{0}^{t}\nabla F_{-}\cdot dX\right)\left(\omega\right)\\
\sum_{k}\left(\nabla F_{t_{k}^{n}}\left(\omega_{\wedge t_{k-1}^{n}}^{n}\right)-\nabla F_{t_{k-1}^{n}}\left(\omega_{\wedge t_{k-1}^{n}}^{n}\right)\right)\cdot X_{t_{k-1}^{n},t_{k}^{n}}\left(\omega\right) & \rightarrow_{n\rightarrow\infty}^{\mathbb{P}} & 0.
\end{eqnarray*}
For the first sum note that $\nabla F_{t_{k-1}^{n}}\left(\omega_{\wedge t_{k-1}^{n}}^{n}\right)=\nabla F_{t_{k-1}^{n}}\left(\omega^{n}\right)$.
For brevity we introduce the processes $I^{n}$, $\left(t,\omega\right)\mapsto I_{t}^{n}\left(\omega\right):=\nabla F_{t}\left(\omega^{n}\right)$
and $I_{t}\left(\omega\right):=\nabla F_{t}\left(\omega\right)$ and
note that $I_{-}^{n}\rightarrow_{n}I_{-}$ in $\mathbb{P}$-probability
(since $\nabla F$ is regular by assumption). For $\mathbb{P}$-a.e.~$\omega$
we have 
\[
\sum_{k}\nabla F_{t_{k-1}^{n}}\left(\omega^{n}\right)\cdot X_{t_{k-1}^{n},t_{k}^{n}}\left(\omega\right)=\left(\int I_{_{n}r}^{n}dX_{r}\right)\left(\omega\right)=\left(\int I_{_{n}r}dX_{r}\right)\left(\omega\right)+\left(\int\left(I_{_{n}r}^{n}-I_{_{n}r}\right)dX_{r}\right)\left(\omega\right)
\]
(recall the notation $r\in\left(\ensuremath{\prescript{}{n}{r}},\ensuremath{\prescript{n}{}{r}}\right]$).
Note that $\left(I_{_{n}r}\right)_{r\in\mathbb{R}_{+}}$ is a piecewise
constant càglàd process (therefore a simple integrand), hence $\int I_{_{n}r}dX_{r}$
converges ucp to $\int_{0}^{t}I_{-}dX\equiv\int_{0}^{t}\nabla F_{-}dX$;
since $I_{-}^{n}$ converges ucp to $I_{-}\equiv\nabla F_{-}$ we
get that$\int\left(I_{_{n}r}^{n}-I_{_{n}r}\right)dX$ converges ucp
to $0$. It remains to estimate 
\begin{equation}
\sum_{k}\left(\nabla F_{t_{k}^{n}}\left(\omega_{\wedge t_{k-1}^{n}}^{n}\right)-\nabla F_{t_{k-1}^{n}}\left(\omega_{\wedge t_{k-1}^{n}}^{n}\right)\right)\cdot X_{t_{k-1}^{n},t_{k}^{n}}\left(\omega\right).\label{eq:diff}
\end{equation}
Therefore note that $\left(\nabla F_{t_{k}^{n}}\left(._{\wedge t_{k-1}^{n}}^{n}\right)-\nabla F_{t_{k-1}^{n}}\left(._{\wedge t_{k-1}^{n}}^{n}\right)\right)\in\mathcal{F}_{t_{k-1}^{n}}^{\mathbb{P}}$,
hence above sum can again be seen as the stochastic integral of a
simple integrand and causal continuity in time implies 
\[
\sup_{k}\left|\nabla F_{t_{k}^{n}}\left(\omega_{\wedge t_{k-1}^{n}}^{n}\right)-\nabla F_{t_{k-1}^{n}}\left(\omega_{\wedge t_{k-1}^{n}}^{n}\right)\right|\leq\rho_{t,2\left|\omega\right|}\left(mesh\left(\pi_{n}\right)+\sup_{k}\left(\omega_{t_{k-1}^{n}}-\omega_{t_{k-2}^{n}}\right)\right)\rightarrow_{n}0\text{ for \ensuremath{\mathbb{P}}-a.e.\,\ensuremath{\omega}}
\]
which is sufficient to conclude that (\ref{eq:diff}) converges to
$0$ in probability. To sum up, we have shown that the sum $\sum_{k}\nabla F_{t_{k+1}^{n}}\left(\omega_{\wedge t_{k-1}^{n}}^{n}\right)\cdot X_{t_{k-1}^{n},t_{k}^{n}}\left(\omega\right)$
appearing in (\ref{eq:S_n_k}) converges ucp to $\int_{0}^{t}\nabla F_{-}dX$.
The very same arguments imply the convergence 
\[
\frac{1}{2}\sum_{k}Tr\left[X_{t_{k-1}^{n},t_{k}^{n}}^{T}\left(\omega\right)\cdot\Delta F_{t_{k}^{n}}\left(\omega_{\wedge t_{k-1}^{n}}^{n}\right)\cdot X_{t_{k-1}^{n},t_{k}^{n}}\left(\omega\right)\right]\rightarrow_{n\rightarrow\infty}^{\mathbb{P}}\int_{0}^{t}Tr\left[\Delta F_{-}\cdot d\left[X\right]\right].
\]

\textbf{Step 2: }$\sum_{k}T_{k}^{n}\rightarrow_{n}^{\mathbb{P}}\int_{0}^{t}\partial_{0}F_{r}dr$\textbf{.}
Causal time differentiability of $F$ implies that 
\[
\mathbb{R}_{+}\ni\bullet\mapsto F_{t_{k-1}^{n}+\bullet}\left(\omega_{\wedge t_{k-1}^{n}}^{n}\right)
\]
is continuous and has a right-derivative on%
\footnote{Not only at $\bullet=0$ since $F_{\left(t_{k-1}^{n}+r\right)+\bullet}\left(\omega_{\wedge t_{k-1}^{n}}^{n}\right)=F_{\left(t_{k-1}^{n}+r\right)+\bullet}\left(\left(\omega_{\wedge t_{k-1}^{n}}^{n}\right)_{\wedge\left(t_{k-1}^{n}+r\right)}\right)$
$\forall r\in\mathbb{R}_{+}$, hence $F_{t_{k-1}^{n}+r+\bullet}\left(\omega_{\wedge t_{k-1}^{n}}^{n}\right)-F_{t_{k-1}^{n}+r}\left(\omega_{\wedge t_{k-1}^{n}}^{n}\right)=\partial_{0}F_{t_{k-1}^{n}+r}\left(\omega_{\wedge t_{k-1}^{n}}^{n}\right)\bullet+o\left(\bullet\right)$
as $\bullet\searrow0$.%
} $\mathbb{R}_{+}$ which is additionally Riemann integrable. This
is enough for a modification of the fundamental theorem of Riemann
integration to apply (see \cite{botsko1986stronger}): 
\begin{eqnarray*}
F_{t_{k-1}^{n}+\left(t_{k}^{n}-t_{k-1}^{n}\right)}\left(\omega_{\wedge t_{k-1}^{n}}^{n}\right)-F_{t_{k-1}^{n}}\left(\omega_{\wedge t_{k-1}^{n}}^{n}\right) & = & \int_{0}^{t_{k}^{n}-t_{k-1}^{n}}\partial_{0}F_{t_{k-1}^{n}+r}\left(\omega_{\wedge t_{k-1}^{n}}^{n}\right)dr.
\end{eqnarray*}
Taking the sum $\sum_{k}$ over all $k$ s.t.~$\pi\left(n\right)\ni t_{k-1}^{n},t_{k}^{n}\leq t^{n}$
gives 
\[
\sum_{k}T_{k}^{n}=\int_{0}^{t^{n}}\partial_{0}F_{r}\left(\omega_{\wedge r_{n}}^{n}\right)dr.
\]
Since $\partial_{0}F$ is causal, $\partial_{0}F_{r}\left(\omega_{\wedge r_{n}}^{n}\right)=\partial_{0}F_{r}\left(\omega^{n}\right)$
for $r\notin\pi\left(n\right)$ (which is a Lebesgue null-set of $\left[0,t\right]$),
we have shown 
\[
\sum_{k}T_{k}^{n}=\int_{0}^{t^{n}}\partial_{0}F_{r}\left(\omega^{n}\right)dr.
\]
By assumption the integrand converges ucp on $\left(\Omega,\mathcal{F}^{\mathbb{P}},\mathcal{F}_{t}^{\mathbb{P}},\mathbb{P}\right)$
to $\partial_{0}F$, hence we have shown the convergence 
\[
\sum_{k}T_{k}^{n}\rightarrow_{n\rightarrow\infty}^{\mathbb{P}}\int_{0}^{t}\partial_{0}F_{r}dr.
\]
\end{proof}
\begin{rem}
\label{Rem: affine linear}Theorem \ref{thm:pathdependent ito} says
that $F\in C^{1,2}$ implies that $F$ is an It\={o}--process with
regular coefficients. In Section \ref{sec:A-regular-and} we show
(Theorem \ref{thm:smooth stoch integral} and Example \ref{ex:regular ito process})
the converse, i.e.~every process of the form $\int_{0}^{t}\mu_{r}dr+\int_{0}^{t}\sigma_{r-}dX_{r}$
has (provided $\mu,\sigma$ are regular enough) an aggregator in $C^{1,2}$.
Theorem \ref{thm:pathdependent ito} immediately gives the explicit
Bichteler--Dellacherie--Meyer decomposition of $F$ under every $\mathbb{P}\in\mathcal{M}_{c}^{loc}\subset\mathcal{M}_{c}^{semi}$.
\begin{rem}
Dupire showed above formula for the case when $\left(t,\omega\right)\mapsto F\left(t,\omega\right)$
is pathwise continuous in the (pseudo-)metric of Example \ref{ex: continuous proc}
and $\mathbb{P}$ is the Wiener measure. Cont and Fournie \cite{ContFournie,ContFournie_Pathwise}
give a mathematical precise argument and two proofs: a pathwise, analytic
proof in \cite{ContFournie_Pathwise} and a probabilistic proof in
\cite{ContFournie}. Further, both apply to functionals which depend
pathwise continuoulsy on $\omega$ and a second component. 
\begin{rem}
Motivated by the case of no pathdependence, i.e.~$F_{t}\left(\omega\right)=f\left(t\right)$
one could argue that a natural criteria for a fundamental theorem
of calculus is absolute continuity instead of right-differentiability.
In fact, alternatively to Definition \ref{def:time-derivative}, one
could say that $F$ has a causal time derivative if $r\mapsto F_{r}\left(\omega_{\wedge t}\right)$
is absolutely continuous and above proof then applies modulo obvious
modifications. 
\begin{rem}
We restricted attention to continuous semimartingale measures $\mathcal{M}_{c}^{semi}$
(i.e.~$t\mapsto X_{t}$ is a.s.~continuous) but under extra assumptions
on the regularity of $F$ one can also in the case when $X$ is a
càdlàg semimartingale subtract the small jumps before steps 1 and
2 (see also \cite{ContFournie_Pathwise}). However, the assumptions
on $F$ get more technical and in view the applications given in the
sections below we decided to stick to the case $\mathcal{M}_{c}^{semi}$
since our focus lies on the lack of continuity $\omega\mapsto F_{.}\left(\omega\right)$.
In the same spirit one could ask for a Tanaka formula etc.
\end{rem}
\end{rem}
\end{rem}
\end{rem}
The proof of Theorem \ref{thm:pathdependent ito} requires only regularity
of $F$ and its derivatives along one sequence of partitions converging
to identity and with respect to a fixed measure $\mathbb{P}\in\mathcal{M}_{c}^{semi}$,
i.e.~we have shown a more general statement: denote with $C_{w}^{1,2}$
the class of functionals which fulfill Definition \ref{def:C12} when
we replace regular with the condition that for every $\mathbb{P}\in\mathcal{M}_{c}^{semi}$
there exists a sequence of partitions $\pi=\left(\pi\left(n\right)\right)$
converging to identity such that $F$ and its derivatives are weakly-regular
along $\pi$.
\begin{cor}
If $F\in C_{w}^{1,2}$ then $F$ is a continuous semimartingale on
$\left(\Omega,\mathcal{F}^{\mathbb{P}},\mathcal{F}_{t}^{\mathbb{P}},\mathbb{P}\right)$
and 
\[
F_{.}-F_{0}=\int_{0}^{.}\partial_{0}F_{r}dr+\sum_{i=1}^{d}\int_{0}^{.}\partial_{i}F_{r-}dX_{r}^{i}+\frac{1}{2}\sum_{i,j=1}^{d}\int_{0}^{.}\partial_{ij}F_{r}d\left[X^{i},X^{j}\right]_{r}\,\,\,\,\mathbb{P}-a.s.
\]

\end{cor}
Applied with $F_{t}\left(\omega\right)=f\left(t,X_{t}\left(\omega\right)\right)$
all the above reduces to the standard It\={o} formula and motivated
by this one can ask for generalizations of the Feynman--Kac formula
or the classic link between (sub-)harmonic functions and (sub-)martingales.
Indeed, all this was already done by Dupire \cite{Duprire_FIto} and
one can follow the same arguments the only difference being that we
now cover a larger class of aggregators -- though we still have to
convince the reader of this, i.e.~that $C^{1,2}$ includes interesting
functionals besides the functionals of Example \ref{ex: continuous proc}.
Before that we briefly come back to the non-uniqueness of the causal
derivatives seen as processes as pointed out in Remark \ref{rem:indistinguishable}.

\subsection{\label{sub:Uniqueness-of-causal}Uniqueness of causal heat operator
and space derivative acting on processes}

Due to the pathwise nature of the causal derivatives, it may happen
that two different functionals give rise to processes that are indistinguishable,
both are differentiable but these derivatives are no longer indistinguishable
processes. 
\begin{example}
\label{ex:not unique}Let $d=1$. In Section \ref{sec:A-regular-and}
we introduce a causal functional, denoted $\left[X\right]^{BK}:\mathbb{R}_{+}\times\Omega\rightarrow\mathbb{R}$,
s.t.~$\left[X\right]^{BK}\in C^{1,2}$ and under every $\mathbb{P}\in\mathcal{M}_{c}^{semi}$,
$\left[X\right]^{BK}$ is indistinguishable on $\left(\Omega,\mathcal{F}^{\mathbb{P}},\mathcal{F}^{\mathbb{P}},\mathbb{P}\right)$
from the quadratic variation process of $X$ on $\left(\Omega,\mathcal{F}^{\mathbb{P}},\mathcal{F}^{\mathbb{P}},\mathbb{P}\right)$.
Consider the two causal functionals, both are in $C^{1,2}$, 
\[
\left(t,\omega\right)\mapsto\left[X\right]_{t}^{BK}\left(\omega\right)\text{ and }\left(t,\omega\right)\mapsto t
\]
which are indistinguishable processes if $X$ is a Brownian motion
under $\mathbb{P}$ however their causal derivatives are distinguishable
under $\mathbb{P}$ (e.g.~$\partial_{0}\left(t\right)=1$ but $\partial_{0}\left(\left[X\right]^{BK}\right)=0$
$\mathbb{P}$-a.s., cf.~Corollary \ref{cor:BK quadr. variation}).
\end{example}
However, the situation is not too bad: in above example the first
spatial derivative $\nabla$ and $\partial_{0}+\frac{1}{2}Tr\left[\Delta\cdot d\left[X\right]\right]$
applied to these functionals are again indistinguishable processes
on $\left(\Omega,\mathcal{F}^{\mathbb{P}},\mathcal{F}^{\mathbb{P}},\mathbb{P}\right)$.
Theorem \ref{thm:pathdependent ito} gives us the simple explanation
that $\partial_{0}F$ and $\frac{1}{2}\Delta F$ both replicate the
behaviour of $F$ on the time-scale of $t$ as opposed to $\nabla F$
which replicates the behaviour of $F$ on the $\sqrt{t}$-scale, hence
we can only expect uniqueness of the ``causal heat operator'' $\partial_{0}F+\frac{1}{2}Tr\left[\Delta Fd\left[X\right]\right]$
and $\nabla F$. 
\begin{prop}
\label{prop:uniqueness}Let $F,G\in C^{1,2}$, $\mathbb{P}\in\mathcal{M}_{c}^{loc}$
and assume $\left[X\right]_{.}=\int_{0}^{.}\sigma^{2}dr$ under $\mathbb{P}$
for some continuous, $\mathcal{F}_{t}^{\mathbb{P}}$-adapted process
$\sigma$. Further assume $F$ and $G$ are indistinguishable on $\left(\Omega,\mathcal{F}^{\mathbb{P}},\mathcal{F}_{t}^{\mathbb{P}},\mathbb{P}\right)$
i.e. 
\[
F_{.}=G_{.}\,\,\,\mathbb{P}-a.s
\]
Then
\begin{equation}
\partial_{0}F+\frac{1}{2}Tr\left[\Delta F\cdot\sigma^{2}\right]=\partial_{0}G+\frac{1}{2}Tr\left[\Delta G\cdot\sigma^{2}\right]\,\,\,\mathbb{P}-a.s.\label{eq:unique causal}
\end{equation}
and 
\begin{equation}
\int_{0}^{.}\nabla F\cdot dX=\int_{0}^{.}\nabla G\cdot dX\,\,\,\mathbb{P}-a.s.\label{eq:unique spatial}
\end{equation}
\end{prop}
\begin{proof}
As in \cite{ContFournie}, this is a simple consequence of the uniqueness
of the semimartingale decomposition: from Theorem \ref{thm:pathdependent ito}
it follows that $H_{t}\left(\omega\right):=F_{t}\left(\omega\right)-G_{t}\left(\omega\right)$
is a continuous semimartingale on $\left(\Omega,\mathcal{F}^{\mathbb{P}},\mathcal{F}_{t}^{\mathbb{P}},\mathbb{P}\right)$
but every continuous semimartingale is a special semimartingale, that
is the decomposition $H=M+A$ into a local martingale $M$ and a process
$A$ of locally bounded variation is unique up to $\mathbb{P}$-indistinguishability
which already implies (\ref{eq:unique spatial}). Since $H$ is indistinguishable
from the process $0$ this implies 
\[
A_{.}=\int_{0}^{.}\left(\partial_{0}H_{r}+\frac{1}{2}Tr\left[\Delta H_{r}\cdot\sigma_{r}^{2}\right]\right)dr
\]
is indistinguishable from $0$, hence we have that $\mathbb{P}$-a.s.~$\partial_{0}H_{t}+\frac{1}{2}Tr\left[\Delta H_{t}\cdot\sigma_{t}^{2}\right]=0$
Lebesgue a.e.~which implies (\ref{eq:unique causal}).\end{proof}
\begin{rem}
Similar path-dependent heat operators appear already in seminal work
of Kusuoka--Stroock \cite{kusuoka1991precise} using Malliavin calculus
techniques on abstract Wiener spaces. However, no Feynman--Kac formula
was given and the abstract Wiener space setting is of course essential
for the results in \cite{kusuoka1991precise}. Further, we mention
the interesting approach of Ahn \cite{0876.60025} on semimartingale
representations via Fréchet derivatives.
\begin{rem}
\label{rem:heat ope}Under stronger assumptions on $\mathbb{P}$ (e.g.~when
$X$ is a non-degenerate Brownian martingale) one can also deduce
the indistinguishability of $\nabla F$ and $\nabla G$ since then
$\int\nabla F\cdot dX=\int\nabla G\cdot dX$ already implies indistinguishability
of $\nabla F$ and $\nabla G$. This is essential to define an extension
of the causal derivatives $\nabla$ and $\partial_{0}+\frac{1}{2}Tr\left[\Delta.d\left[X\right]\right]$
which are defined pathwise to operators on (equivalence classes of)
It\={o}--processes via closure available under a fixed measure as
in Malliavin calculus cf.~\cite[Section 5]{ContFournie} and Remark
\ref{rem:closure1}. We note that in classic work of Kusuoka--Stroock
\cite{kusuoka1991precise} a similar phenomenon appears where the
Fréchet derivative $D$ and only the ``Malliavin heat operator'' $\mathcal{A}=\frac{\partial}{\partial t}+\frac{1}{2}\boldsymbol{T}_{H}D^{2}$
(not only $\frac{\partial}{\partial t}$ or $\frac{1}{2}\boldsymbol{T}_{H}D^{2}$)
are extended by continuity and closure on abstract Wiener spaces (cf.~\cite[p11, Remark 1.17ff]{kusuoka1991precise}).
\end{rem}
\end{rem}

\section{\label{sec:A-regular-and}The it\={o}-integral and quadratic variation
process as causal functionals}

In this section we show that one can construct a pathwise version
of the stochastic integral of a càdlàg process against the coordinate
process $X$ which is an element of $C^{1,2}$, i.e.~a regular, differentiable
causal functional and similarly for the quadratic variation. To achieve
this we make use of the path-wise It\={o}-integral and quadratic variation
as introduced by Bichteler in 1981 using adapted Riemann sums and
subsequent simplifications by Karandikar \cite{bichteler1981stochastic,Karandikar199511}. 
\begin{example}
To appreciate Bichteler's pathwise integral let us briefly recall
why pathwise constructions are non-trivial on the simpler example
of the quadratic variation process in the one-dimensional case ($d=1$).
Once a probability measure $\mathbb{P}\in\mathcal{M}_{c}^{semi}$
and a sequence of partitions $\left(\pi\left(n\right)\right)_{n}$
is fixed, standard results guarantee the existence of a process $\left[X\right]$
on $\left(\Omega,\mathcal{F}^{\mathbb{P}},\mathcal{F}_{t}^{\mathbb{P}},\mathbb{P}\right)$
(for brevity let $d=1$) such that 
\[
X_{0}^{2}+\sum_{k:t_{k-1}^{n}\in\pi\left(n\right)}\left|X_{t_{k}^{n}\wedge t}-X_{t_{k-1}^{n}\wedge t}\right|^{2}\rightarrow_{n}\left[X\right]_{t}\,\,\,\text{uniformly on compacts in \ensuremath{\mathbb{P}}-probability.}
\]
Hence, a candidate for a pathwise version (which is also causally
differentiable) of the quadratic variation process is the map 
\[
\left(t,\omega\right)\mapsto F_{t}\left(\omega\right)=\begin{cases}
\left|\omega\left(0\right)\right|^{2}+\lim_{n\rightarrow\infty}\sum_{k:t_{k-1}^{n}\in\pi\left(n\right)}\left|\omega\left(t_{k}^{n}\wedge t\right)-\omega\left(t_{k-1}^{n}\wedge t\right)\right|^{2} & ,\text{ if this limit exists}\\
0 & ,\text{ otherwise}.
\end{cases}
\]
Indeed, if $\mathbb{P}\in\mathcal{M}_{c}^{semi}$ is such that the
coordinate process $X$ is a Brownian motion and $\left(\pi\left(n\right)\right)_{n}$
are dyadics then above $F$ is a process on $\left(\Omega,\mathcal{F}^{\mathbb{P}},\mathcal{F}_{t}^{\mathbb{P}},\mathbb{P}\right)$
indistinguishable from $\left[X\right]$. However, this is not true
under every $\mathbb{P}\in\mathcal{M}_{c}^{semi}$ or any sequence
of partitions $\left(\pi\left(n\right)\right)_{n}$ converging to
the identity --- in fact, $\sum_{k}\left|\omega\left(t_{k}^{n}\wedge t\right)-\omega\left(t_{k-1}^{n}\wedge t\right)\right|^{2}$
converges in general only in $\mathbb{P}$-probability. (Of course,
for every given $\mathbb{P}\in\mathcal{M}_{c}^{semi}$ one can always
find a fast enough vanishing sub-sequence of $\left(\pi\left(n\right)\right)_{n}$
to ensure convergence $\mathbb{P}$-a.s.) The same problem appears,
afortiori, for the stochastic integral where path-regularity of the
integrand has to be considered as well, etc. 
\end{example}
The main result of this section is Theorem \ref{thm:smooth stoch integral}
below which shows that the causal space derivative acts as the inverse
of a pathwise stochastic integral.
\begin{thm}
\label{thm:smooth stoch integral}Assume $Z$ is a \textup{\emph{real-valued
càdlàg}} process on $\left(\Omega,\mathcal{F}_{t}^{0},\mathcal{F}^{0}\right)$
\textup{\emph{which is càdlàg}} regular. Then for every $i\in\left\{ 1,\ldots,d\right\} $
there exists a regular \textup{\emph{càdlàg}} functional 
\[
J_{Z}^{i}:\mathbb{R}_{+}\times\Omega\rightarrow\mathbb{R}
\]
 such that
\begin{enumerate}
\item $J_{Z}^{i}$ has causal derivatives $\partial_{0}J_{Z}^{i}$,$\nabla J_{Z}^{i}$,$\Delta J_{Z}^{i}$
and these are again regular functionals,
\item $J_{Z}^{i}$ is under every $\mathbb{P}\in\mathcal{M}_{c}^{semi}$
indistinguishable from the It\={o}-integral $\int_{0}^{.}Z_{-}dX^{i}$
on $\left(\Omega,\mathcal{F}^{\mathbb{P}},\mathbb{\mathcal{F}}_{t}^{\mathbb{P}},\mathbb{P}\right)$,
i.e. 
\[
J_{Z}^{i}\left(.,\omega\right)=\left(\int_{0}^{.}Z_{-}dX^{i}\right)\left(\omega\right)\text{ for \ensuremath{\mathbb{P}}-a.e. }\omega.
\]

\end{enumerate}
Moreover $\forall\left(t,\omega\right)\in\mathbb{R}_{+}\times\Omega$,
$t<\zeta\left(\omega\right)$, 
\begin{eqnarray*}
\partial_{0}J_{Z}^{i}\left(t,\omega\right) & = & 0,\\
\nabla J_{Z}^{i}\left(t,\omega\right) & = & \left(0,\ldots0,Z_{t-}\left(\omega\right),0,\ldots0\right)^{T}\text{ (only \ensuremath{i^{th}}component non-zero),}\\
\Delta J_{Z}^{i}\left(t,\omega\right) & = & 0
\end{eqnarray*}
and if $Z_{-}$ is causally continuous in time then $J_{Z}^{i}\in C^{1,2}$.
We also use the notation $\int_{0}^{.}Z_{-}\left(\omega\right)d^{BK}X^{i}\left(\omega\right)$
for the process $J_{Z}^{i}\left(\omega\right)$ ($d^{BK}$ stands
for Bichteler--Karandikar), resp.~in the multidimensional case $Z=\left(Z^{i}\right)_{i=1}^{d}$
we write $\int_{0}^{.}Z_{-}\left(\omega\right)\cdot d^{BK}X\left(\omega\right)$
for $\sum_{i=1}^{d}J_{Z^{i}}^{i}\left(\omega\right)$.
\end{thm}

\begin{rem}
\label{rem:closure1}In \cite[Section 5]{ContFournie} $\nabla$ is
extended under a fixed measure $\mathbb{P}$ to an operator $\nabla_{X}$
acting on square--integrable (standard) stochastic\emph{ }It\={o}-integrals
which can be identified via the Clark--Hausmann--Ocone Theorem as
the predictable projection of the Malliavin derivative. We refer the
reader to \cite[Section 5]{ContFournie}, but remark that Theorem
\ref{thm:smooth stoch integral} shows that the causal derivative
$\nabla$ acts as the \emph{pathwise} \emph{inverse} to an aggregator
of the It\={o}-integral, e.g.~in the case of a Wiener functional
$\xi\in L^{2}\left(\mathcal{F}_{T}^{\mathbb{P}}\right)$ that has
the representation $\xi=\mathbb{E}_{\mathbb{P}}\left[\xi\right]+\int_{0}^{T}Z_{-}\cdot dX$
with $Z$ as in Theorem \ref{thm:smooth stoch integral} then $\xi$
has a representation as $\xi_{T}^{0}$, $\xi^{0}\in C^{1,2}$, given
as $\xi_{t}^{0}\left(\omega\right):=\mathbb{E}_{\mathbb{P}}\left[\xi^{0}\right]+\int_{0}^{t}\nabla\xi_{-}^{0}\left(\omega\right)\cdot d^{BK}X\left(\omega\right)$;
more interestingly this is not restricted to Wiener functionals provided
of course the existence of such a representation $\xi=m+\int_{0}^{T}Z_{-}\cdot dX$.
The considerable price is that already in the Brownian context one
has to assume some regularity of $Z$ but martingale representation
does not guarantee any such regularity. However it might be a reasonable
assumption for applications (e.g.~implementable Delta-hedging strategies,
integrands of the form $Z_{t}=v\left(t,X_{t},\left[X\right]_{t}^{BK}\right)$
etc; note also that using the non-pathwise $L^{2}$-extension $\nabla_{X}$
leads to not completely trivial issues about consistent approximations
by the pathwise operator $\nabla$ acting on cylinder functions).
\end{rem}
The process $J_{Z}^{i}$ in Theorem \ref{thm:smooth stoch integral}
is constructed via the pathwise Bichteler--Karandikar integral which
we briefly recall.
\begin{thm}[{Bichteler \cite{bichteler1981stochastic},Karandikar \cite[Theorem 2 and 3]{Karandikar199511}}]
\label{thm:karandikar}There exists a map 
\[
I:D\left(\mathbb{R}_{+},\mathbb{R}\right)\times D\left(\mathbb{R}_{+},\mathbb{R}\right)\rightarrow D\left(\mathbb{R}_{+},\mathbb{R}\right)
\]
s.t.~for any filtered probability space $\left(\tilde{\Omega},\tilde{\mathcal{F}},\left(\mathcal{\tilde{F}}_{t}\right),\mathbb{\tilde{P}}\right)$
satisfying the usual conditions, any real-valued semimartingale $\tilde{X}$
and any adapted, real-valued, càdlàg process $\tilde{Z}$ on $\left(\tilde{\Omega},\tilde{\mathcal{F}},\left(\mathcal{\tilde{F}}_{t}\right),\mathbb{\tilde{P}}\right)$
we have 
\[
I\left(\tilde{Z}\left(\omega\right),\tilde{X}\left(\omega\right)\right)\left(.\right)=\left(\int_{0}^{.}\tilde{Z}_{-}d\tilde{X}\right)\left(\omega\right)\text{ \ensuremath{\mathbb{\tilde{P}}}-a.s.}
\]
\end{thm}
\begin{rem}
The set of full $\mathbb{\tilde{P}}$-measure on which the equality
holds depends of course on the processes $\tilde{Z}$ and $\tilde{X}$. 
\begin{rem}
Another well-known result on pathwise integration is Föllmer's It\={o}-integral
``sans probabilités'' \cite{follmer-81}. However, this allows only
to give a pathwise version when the integrand is of the type $\int\nabla f\left(X_{s}\right)\cdot dX_{s}$,
i.e.~a non-pathdependent integrand which is in gradient form (the
latter is a strong restriction in multi-dimensions $d>1$, cf.~Example
\ref{ex: Levy-area} for a process which is not of this type). The
non-pathdependence can be relaxed, see \cite{ContFournie_Pathwise},
however the gradient form is essential for Föllmer's argument. Let
us also mention the approach via capacity as in Denis and Martini
\cite{denis2006theoretical} to construct an integral for quasi-sure
integrands and Peng's approach via PDEs \cite{2007arXiv0711.2834P}
which lead to a much less explicit integral but on the other hand
apply to less regular integrands than the Bichteler---Karandikar integral.
\begin{rem}
\label{rem:bichteler}The map $I$ is constructed as follows: for
$x,z\in D\left(\mathbb{R}_{+},\mathbb{R}\right)$ and $n\in\mathbb{N}$
define the sequence $\left(\tau_{i}^{n}\right)_{i\geq0}\subset\mathbb{R}_{+}\cup\left\{ +\infty\right\} $
recursively as%
\footnote{With the convention $\inf\emptyset=\infty$ and $\tau_{i+1}^{n}=\infty$
if $\tau_{i}^{n}=\infty$.%
} 
\[
\tau_{i}^{n}=\inf\left\{ t>\tau_{i-1}^{n}:\left|z_{t}-z_{\tau_{i-1}^{n}}\right|\geq2^{-n},t\in\mathbb{R}_{+}\right\} \text{ and }\tau_{0}^{n}\equiv0.
\]
Define $I^{n}:D\left(\mathbb{R}_{+},\mathbb{R}\right)\times D\left(\mathbb{R}_{+},\mathbb{R}\right)\rightarrow D\left(\mathbb{R}_{+},\mathbb{R}\right)$
as 
\begin{equation}
I^{n}\left(z,x\right)\left(t\right)=z\left(0\right)x\left(0\right)+\sum_{i\geq0}z\left(\tau_{i}^{n}\right)\left(x\left(\tau_{i+1}^{n}\wedge t\right)-x\left(\tau_{i}^{n}\wedge t\right)\right)\label{eq:BK_finite sum}
\end{equation}
and set $I\left(z,x\right)=\lim_{n}I^{n}\left(z,x\right)$ whenever
this limit exists in the topology of uniform convergence on compact
sets on $D\left(\mathbb{R}_{+},\mathbb{R}\right)$, otherwise set
$I\left(z,x\right)=0$.
\end{rem}
\end{rem}
\end{rem}

\begin{proof}[Proof of Theorem \ref{thm:smooth stoch integral}]
We define 
\begin{equation}
J_{Z}^{i}\left(t,\omega\right):=\begin{cases}
I\left(Z\left(\omega\right),X^{i}\left(\omega\right)\right)_{t} & ,\text{if}\, t<\zeta\left(\omega\right)\\
0 & \text{,else}
\end{cases}\label{eq:def J}
\end{equation}
where $I\left(.,.\right):D\left(\mathbb{R}_{+},\mathbb{R}\right)\times D\left(\mathbb{R}_{+},\mathbb{R}\right)\rightarrow D\left(\mathbb{R}_{+},\mathbb{R}\right)$
is the map constructed in Theorem \ref{thm:karandikar} and Remark
\ref{rem:bichteler}. For brevity write $J$ instead of $J_{Z}^{i}$
throughout this proof. Point $\left(2\right)$ follows immediately
from Theorem \ref{thm:karandikar}. 

To see the causal differentiability, consider $\left(t,\omega\right)\in\mathbb{R}_{+}\times\Omega$
with $t<\zeta\left(\omega\right)$. In this case, if the adapted Riemann
sum (\ref{eq:BK_finite sum}) of the Bichteler--Karandikar integral
$I\left(Z\left(\omega\right),X{}^{i}\left(\omega\right)\right)\left(t\right)$
converges as a uniform limit then this is also the case for $\omega^{t,r}$
or $\omega_{\wedge t}$ for any $r=\left(r^{1},\ldots,r^{d}\right)\in\mathbb{R}^{d}$:
first note that $Z$ is $\mathcal{O}^{0}$-optional (since $\mathcal{O}^{0}$
is generated by the càdlàg processes) and so by Proposition \ref{prop:canonical pathspace}
the trajectories of $Z\left(\omega\right)$ and $Z\left(\omega_{\wedge t}\right)$
coincide on $\left[0,t\right]$, hence the sequence $\left(\tau_{i}^{n}\left(\omega\right)\right)_{i\geq0}$
is the same as $\left(\tau_{i}^{n}\left(\omega_{\wedge t}\right)\right)_{i\geq0}$
when restricted to $\left[0,t\right]$. Then directly from the definition
of the Bichteler--Karandikar integral as limit of (\ref{eq:BK_finite sum})
it follows that 
\[
J\left(t+\epsilon,\omega_{\wedge t}\right)=I\left(Z\left(\omega_{\wedge t}\right),X^{i}\left(\omega_{\wedge t}\right)\right)_{t+\epsilon}=I\left(Z\left(\omega_{\wedge t}\right),X^{i}\left(\omega\right)\right)_{t}=I\left(Z\left(\omega\right),X^{i}\left(\omega\right)\right)_{t},
\]
which implies $\partial_{0}J\left(t,\omega\right)=0$. Similarly,
we can use that the trajectories of $Z\left(\omega^{t,r}\right)$
and $Z\left(\omega\right)$ coincide on $\left[0,t\right)$ to see
that 
\begin{eqnarray*}
J\left(t,\omega^{t,r}\right) & = & I\left(Z\left(\omega^{t,r}\right),X^{i}\left(\omega^{t,r}\right)\right)_{t}=I\left(Z\left(\omega^{t,r}\right),X^{i}\left(\omega\right)\right)_{t}+Z_{t-}\left(\omega^{t,r}\right)r^{i}\\
 & = & I\left(Z\left(\omega\right),X^{i}\left(\omega\right)\right)_{t}+Z_{t-}\left(\omega\right)r^{i}
\end{eqnarray*}
which shows the claimed first space derivative. By $\mathcal{O}^{0}$-measurability
(resp.~causality) of $Z$ 
\[
Z_{t-}\left(\omega^{t,\Delta}\right)=\lim_{r\nearrow t}Z_{r}\left(\omega^{t,\Delta}\right)=\lim_{r\nearrow t}Z_{r}\left(\left(\omega^{t,\Delta}\right)_{\wedge r}\right)=Z_{t-}\left(\omega^{t,-\Delta_{t}\omega}\right)
\]
and similarly $Z_{t-}\left(\omega\right)=Z_{t-}\left(\omega^{t,-\Delta_{t}\omega}\right)$
which implies $\Delta J_{t}\left(\omega\right)=0$. 

To see that $J\in C^{1,2}$ we first show that $J$ is a continuous,
regular functional, more precisely that $\forall\mathbb{P}\in\mathcal{M}_{c}^{semi}$
the map $\left(t,\omega\right)\mapsto J_{t}\left(\omega^{n}\right)$
is an adapted càdlàg process that converges ucp to the ($\mathbb{P}$-a.s.)
continuous process $\left(t,\omega\right)\mapsto J_{t}\left(\omega\right)$
on $\left(\Omega,\mathcal{F}^{\mathbb{P}},\mathcal{F}_{t}^{\mathbb{P}},\mathbb{P}\right)$.
Therefore fix $\mathbb{P}\in\mathcal{M}_{c}^{semi}$, any sequence
of partitions $\left(\pi\left(n\right)\right)_{n}$ converging to
the identity and denote the associated piecewise constant approximation
of $\omega$ with $\left(\omega^{n}\right)_{n}$. From the very definition
of $J$ we have $\forall\left(t,\omega\right)\in\mathbb{R}_{+}\times\Omega$
\[
J_{t}^{n}\left(\omega\right):=J_{t}\left(\omega^{n}\right)=I\left(Z\left(\omega^{n}\right),X^{i}\left(\omega^{n}\right)\right)\left(t\right).
\]
Now, 
\[
\left(t,\omega\right)\mapsto Z_{t}^{n}\left(\omega\right):=Z_{t}\left(\omega^{n}\right)\text{ and }\left(t,\omega\right)\mapsto X_{t}^{n}\left(\omega\right):=X_{t}^{i}\left(\omega^{n}\right)
\]
are càdlàg processes on $\left(\Omega,\mathcal{F}^{\mathbb{P}},\mathcal{F}_{t}^{\mathbb{P}},\mathbb{P}\right)$
and $X^{n}$ is of bounded variation (on compacts) and therefore a
semimartingale, hence Theorem \ref{thm:karandikar} guarantees that
\[
J_{.}^{n}\left(\omega\right)=\left(\int_{0}^{.}Z_{-}^{n}dX^{n}\right)\left(\omega\right)\,\,\text{for }\mathbb{P}-a.e.\,\omega.
\]
so especially $J^{n}$ is a càdlàg process. To see that $J_{-}^{n}\rightarrow J_{-}$
ucp note that for every $n$ we have \noun{$\mathbb{P}$-}a.s. 
\begin{eqnarray*}
\int_{0}^{.}Z_{-}^{n}dX^{n} & = & \int_{0}^{.}\left(Z_{-}^{n}-Z_{-}\right)dX^{n}+\int_{0}^{.}Z_{-}dX^{n}\\
 & = & \sum_{k:t_{k}^{n}\leq\cdot}\left(Z_{-}^{n}-Z_{-}\right)_{t_{k-1}^{n}}X_{t_{k-1}^{n},t_{k}^{n}}+\sum_{k:t_{k}^{n}\leq\cdot}Z_{t_{k-1}^{n}-}X_{t_{k-1}^{n},t_{k}^{n}}\\
 & = & \int_{0}^{.}C^{n}dX+\sum_{k:t_{k}^{n}\leq\cdot}Z_{t_{k-1}^{n}-}X_{t_{k-1}^{n},t_{k}^{n}}
\end{eqnarray*}
with $C_{t}^{n}=\sum_{\pi\left(n\right)}\left(Z_{t_{k}^{n}-}^{n}-Z_{t_{k}^{n}-}\right)1_{\left[t_{k}^{n}\wedge.,t_{k+1}^{n}\wedge.\right)}$
(the second equality holds since $X^{n}$ is of bounded variation
on compacts, hence the stochastic integrals are indistinguishable
from pathwise Lebesgue integrals, cf.~\cite[Theorem 18]{MR2020294}).
Since $Z$ is a regular functional, $C^{n}\rightarrow0$ ucp and the
first integral converges to $0$ ucp as $n\rightarrow\infty$. Standard
results show ucp convergence of the Riemann sums to $\int_{0}^{.}Z_{-}dX$.
\end{proof}
As an immediate but important corollary we get the existence of a
$C^{1,2}$ version of the quadratic variation $\left[X\right]$ of
$X$. 
\begin{cor}
\label{cor:BK quadr. variation}For every $i,j\in\left\{ 1,\ldots,d\right\} $
there exists a causal functional 
\[
B^{ij}:\mathbb{R}_{+}\times\Omega\rightarrow\mathbb{R}
\]
 such that 
\begin{enumerate}
\item $B^{ij}\in C^{1,2}$,
\item $B^{ij}$ is indistinguishable from $\left[X^{i},X^{j}\right]$ on
$\left(\Omega,\mathcal{F}^{\mathbb{P}},\mathcal{F}_{t}^{\mathbb{P}},\mathbb{P}\right)$
for every \emph{$\mathbb{P}\in\mathcal{M}_{c}^{semi}$}, i.e.~
\[
B_{.}^{ij}=\left[X^{i},X^{j}\right]_{.}\text{ \ensuremath{\mathbb{P}} -a.s.}
\]
 
\end{enumerate}
Moreover $\forall\left(t,\omega\right)\in\mathbb{R}_{+}\times\Omega$,
$t<\zeta\left(\omega\right)$, $i,j\in\left\{ 1,\ldots,d\right\} $
\begin{eqnarray*}
\partial_{0}B_{t}^{ij}\left(\omega\right) & = & 0
\end{eqnarray*}
and 
\begin{eqnarray*}
\nabla B_{t}^{ij}\left(\omega\right) & = & \left(\Delta_{t}X^{j}\left(\omega\right)1_{k=i}+\Delta_{t}X^{i}\left(\omega\right)1_{k=j}\right)_{k=1,\ldots,d}\\
\Delta B_{t}^{ij}\left(\omega\right) & = & \left(1_{l=j}1_{k=i}+1_{l=i}1_{k=j}\right)_{k,l=1,\ldots d}
\end{eqnarray*}
We also use the notation $\left[X^{i},X^{j}\right]^{BK}$ for $B^{ij}$.\end{cor}
\begin{proof}
Since $X^{i}$ is the coordinate process it fulfills the assumptions
for the integrand in Theorem \ref{thm:smooth stoch integral} except
that $X^{i}$ is not real-valued (only $\mathbb{P}$-a.s.). Therefore
set $\overline{X}_{t}\left(\omega\right)=X_{t}\left(\omega\right)1_{\zeta\left(\omega\right)=\infty}$.
Hence, Theorem \ref{thm:smooth stoch integral} allows us to the define
\begin{equation}
B_{t}^{ij}\left(\omega\right)=\overline{X}_{t}^{i}\left(\omega\right)\overline{X}_{t}^{j}\left(\omega\right)-\int_{0}^{t}\overline{X}_{-}^{i}\left(\omega\right)d^{BK}X^{j}\left(\omega\right)-\int_{0}^{t}\overline{X}_{-}^{j}\left(\omega\right)d^{BK}X^{i}\left(\omega\right)\label{eq:b_ij}
\end{equation}
Points $\left(1\right)\&\left(2\right)$ follow from Theorem \ref{thm:karandikar},
Theorem \ref{thm:smooth stoch integral} and the identity 
\[
\left[X^{i},X^{j}\right]=X^{i}X^{j}-\int_{0}X^{i}dX^{j}-\int_{0}X^{j}dX^{i}.
\]
Finally, apply the causal derivatives to $B^{ij}\left(t,\omega\right)$
which gives for the time derivative 
\begin{eqnarray*}
\partial_{0}\left(X^{i}X^{j}\right)_{t}\left(\omega\right) & = & \partial_{0}\left(\int_{0}^{.}X^{i}\left(\omega\right)d^{BK}X^{j}\left(\omega\right)\right)_{t}=\partial_{0}\left(\int_{0}^{.}X^{j}\left(\omega\right)d^{BK}X^{i}\left(\omega\right)\right)_{t}=0.
\end{eqnarray*}
and from Theorem \ref{thm:smooth stoch integral} the space derivatives
of the integrals follows as well 
\[
\partial_{i}B_{t}^{ij}\left(\omega\right)=\overline{X}_{t}^{j}\left(\omega\right)-\overline{X}_{t-}^{j}\left(\omega\right)=\Delta_{t}\overline{X}^{j}\left(\omega\right).
\]
\end{proof}
\begin{rem}
The first spatial derivative of $B^{ij}$ is indistinguishable under
every continuous semimartingale measure from $0$ but the second spatial
derivative is not! This is a direct consequence of the pathwise nature
of the causal derivatives as pointed out in Remark \ref{rem:indistinguishable}
(pathwise the first derivative is not equal to $0$).
\begin{rem}
It would be desirable to have a result which guarantees the existence
of a $C^{1,2}$-aggregator, i.e.~given a reasonably nice subset $\mathcal{P}$
of $\mathcal{M}_{c}^{semi}$ and a family of processes 
\[
\left\{ Z^{\mathbb{P}}:\mathbb{P\in}\mathcal{P},Z^{\mathbb{P}}\text{ is a progr.\,\ measurable process on }\left(\Omega,\mathcal{F}^{\mathbb{P}},\mathcal{F}_{t}^{\mathbb{P}},\mathbb{P}\right)\right\} 
\]
which fulfill some consistency condition then there exists a causal
functional $Z\in C^{1,2}$ s.t.~$Z=Z^{\mathbb{P}}$ $\mathbb{P}$-a.s.~$\forall\mathbb{P}\in\mathcal{P}$.
Indeed, the only result we are aware of in this direction is the main
result in \cite[Theorem 5.1]{EJP2011-67} which shows for an important
subset $\mathcal{P}$ of $\mathcal{M}_{c}^{semi}$ the existence of
an aggregator. However, even with the restriction to $\mathcal{P}$
it is not clear which conditions would guarantee that the aggregator
is causally differentiable which is why we have choosen a constructive
approach for the two examples of stochastic integration and quadratic
variation which also applies to all elements of $\mathcal{M}_{c}^{semi}$.
\end{rem}
\end{rem}
We give some applications of Theorem \ref{thm:smooth stoch integral}
and Corollary \ref{cor:BK quadr. variation}: the first is an easy
application of the chain and product rule.
\begin{example}
\label{ex: exponential}Let $d=1$ and $E_{t}\left(\omega\right)=\exp\left(X_{t}\left(\omega\right)-\frac{1}{2}\left[X\right]_{t}^{BK}\left(\omega\right)\right)$.
Then $\partial_{0}E_{t}\left(\omega\right)=0$,$\nabla E_{t}\left(\omega\right)=E_{t}\left(\omega\right)\left(1-\Delta_{t}X\left(\omega\right)\right)$
and $\Delta E_{t}\left(\omega\right)=E_{t}\left(\omega\right)\Delta_{t}X\left(\omega\right)\left(-2+\Delta_{t}X\left(\omega\right)\right)$.
Hence, Theorem \ref{thm:pathdependent ito} reduces to the classic
identity 
\[
E_{.}=1+\int_{0}^{.}E_{r}dX\text{ on }\left(\Omega,\mathcal{F}^{\mathbb{P}},\mathcal{F}_{t}^{\mathbb{P}},\mathbb{P}\right),\,\,\mathbb{P}\in\mathcal{M}_{c}^{semi}.
\]

\end{example}
The path-dependent It\={o}-formula, Theorem \ref{thm:pathdependent ito},
shows that every element of $C^{1,2}$ is a It\={o}-process $\mathcal{M}_{c}^{semi}$-quasi-sure.
Using Theorem \ref{thm:karandikar} we can now go the other direction:
\begin{example}
\label{ex:regular ito process}Assume $\mu,\sigma:\mathbb{R}\times\Omega\rightarrow\mathbb{R}$
fulfill the same assumptions as $Z$ in Theorem \ref{thm:smooth stoch integral}\noun{.
T}hen the It\={o}-process 
\[
\int_{0}^{.}\mu_{r}dr+\int_{0}^{.}\sigma_{r-}\cdot dX_{r}=\int_{0}^{.}\mu_{r}dr+\sum_{i=1}^{d}\left(\int_{0}^{.}\sigma_{r-}^{i}dX_{r}^{i}\right)
\]
is indistinguishable on $\left(\Omega,\mathcal{F}^{\mathbb{P}},\mathcal{F}_{t}^{\mathbb{P}},\mathbb{P}\right)$,\noun{
$\mathbb{P}\in\mathcal{M}_{c}^{semi}$, }from $F\in C^{1,2}$ defined
as 
\[
F_{t}\left(\omega\right):=\int_{0}^{t}\mu_{r}\left(\omega\right)dr+\sum_{i=1}^{d}\int_{0}^{t}\sigma_{r-}^{i}\left(\omega\right)d^{BK}X_{r}^{i}\left(\omega\right)
\]
and $\forall\left(t,\omega\right)\in\mathbb{R}_{+}\times\Omega$,$t<\zeta\left(\omega\right)$,
\[
\partial_{0}F_{t}\left(\omega\right)=\mu_{t}\left(\omega\right),\nabla F_{t}\left(\omega\right)=\sigma_{t-}\left(\omega\right)=\left(\sigma_{t-}^{d}\left(\omega\right),\ldots,\sigma_{t-}^{d}\left(\omega\right)\right)^{T},\Delta F_{t}\left(\omega\right)=0.
\]
(hence $\partial_{0}F=\mu$, $\nabla F=\sigma_{-}$ $\mathbb{P}$-a.s.)
which captures the intuition that the causal time derivative measures
the infinitesimal drift and the causal space derivative measures sensitivity
to instantaneous changes of the underlying process $X$, cf.~Remark
\ref{Rem: affine linear}. 
\end{example}
Even if we are only interested in a fixed measure $\mathbb{P}\in\mathcal{M}_{c}^{semi}$,
Theorem \ref{thm:pathdependent ito} and the results of this section
are a non-trivial extension of the functional It\={o}-formula: while
not the topic of this article, we finish with an application of Corollary
\ref{cor:BK quadr. variation} to mathematical finance which allows
to compare it with the approach of Cont--Fournie \cite{ContFournie,ContFournie_Pathwise}
who treat the same example (see \cite[Section 5]{Fourniethesis} for
more applications in finance relying on such a generalized functional
It\={o}-formula).
\begin{example}
For fixed $\mathbb{P}\in\mathcal{M}_{c}^{semi}$ the coordinate process
$X=\left(X^{i}\right)_{i=1}^{d}$ is interpreted as the discounted
price process of $d$ assets in a security market (for brevity we
let $d=1$ and assume no interest rates). Assume there exists an equivalent
martingale measure $\mathbb{Q}\in\mathcal{M}_{c}^{loc}$,$\mathbb{Q}\sim\mathbb{P}$
and $\left[X\right]=\int\sigma_{r}^{2}dr$ under $\mathbb{Q}$ for
some adapted, continuous, bounded process $\sigma^{2}$. The usual
Black--Scholes--Merton dynamic replication argument translates and
shows that if $F\in C^{1,2}$ and 
\[
\partial_{0}F+\frac{\sigma^{2}}{2}\Delta F=0,\, F_{T}=\pi\,\,\mathbb{P}-a.s.
\]
then $F_{t}$ is an arbitrage free price for the claim $\pi$ at time
$t\in\left[0,T\right]$ (i.e.~a version of $\mathbb{E}_{\mathbb{Q}}\left[\pi\lvert\mathcal{F}_{t}^{\mathbb{P}}\right]$)%
\footnote{If $\sigma_{t}^{2}=\sigma_{BS}^{2}X_{t}^{2}$, $\sigma_{BS}^{2}\in\mathbb{R}$
fix and $\pi=\left(X_{T}-K\right)^{+}$ one makes the Ansatz $F_{t}=f\left(t,X_{t}\right)$
and this reduces to the classic Black--Scholes--Merton PDE. The usual
\emph{Greeks} $\partial_{0}F,\nabla F,\Delta F$ are now functionals
and it is interesting to recall Proposition \ref{prop:uniqueness}
and Remark \ref{rem:heat ope} which show that as a process, only
$\partial_{0}F+\frac{\sigma^{2}}{2}\Delta F$ is uniquely defined,
i.e.~in a financial context it makes sense to think of this non-uniqueness
as the Gamma--Theta or convexity--time-decay tradeoff.%
}. Applied to variance swaps we can make the Ansatz $\pi=\left[X\right]_{T}^{BK}-k$
and $F_{t}\left(\omega\right)=\left[X\right]_{t}^{BK}\left(\omega\right)-k+v\left(t,X_{t}\left(\omega\right)\right)$
and above reduces to the parabolic SPDE
\[
\frac{dv}{dt}+\frac{\sigma_{t}^{2}}{2}\frac{d^{2}v}{dx^{2}}=-\sigma_{t}^{2},\, v\left(T,X_{T}\right)=0\,\,\,\mathbb{P}-a.s
\]
(for a local vol.~model $\sigma_{t}^{2}=\sigma_{0}^{2}\left(t,X_{t}\right)X_{t}^{2}$
this further reduces to the backward heat equation with source term
$x^{2}\sigma_{0}\left(t,x\right)$). Of course, more interesting applications
exist for more complex payoffs (nonlinear options on variance, etc.)
in combination with the uncertainty $\mathbb{P}\in\mathcal{P}$ for
$\mathcal{P}\subset\mathcal{M}_{c}^{semi}$ but the related pathdependent
PDEs and questions of uniqueness, existence, interplay with the nonlinear
expectation $\mathcal{E}\left(.\right)=\inf_{\mathbb{P}\in\mathcal{P}}\mathbb{E}_{\mathbb{P}}\left[.\right]$
etc.~touch the tip of an iceberg which is currently under heavy development
even for the case when $\pi$ is a continuous functional in uniform
norm of the underlying --- we draw attention to work of Peng--Wang
\cite{Peng2011,2011arXiv1108.4317P} and work of Ekren--Keller--Touzi--Zhang
\cite{2011arXiv1109.5971E,2012arXiv1210.0006E,2012arXiv1210.0007E}.\end{example}
\begin{acknowledgement}
HO is grateful for support from the European Research Council under
the European Union's Seventh Framework Programme (FP7/2007-2013)/ERC
grant agreement nr.~258237 and DFG Grant SPP-1324. Further, the author
would like to thank Rama Cont for helpful remarks.

\end{acknowledgement}

\bibliographystyle{plain}
\bibliography{/home/hd/Dropbox/projects/BibteX/roughpaths}

\end{document}